\documentclass[11pt]{article}
\usepackage{myarticle-macrosShared}
\usepackage{graphicx,subfig,enumerate,bm}
\usepackage{geometry}
\usepackage{color}

\usepackage{setspace}

\newcommand{\cleanImage}{S}
\newcommand{\noisyImage}{I}
\newcommand{\numberPixels}{p}
\newcommand{\transform}{\mathsf{O}}
\newcommand{\transformkDims}{\mathtt{O}} 
\newcommand{\setOfTransforms}{{\cal T}}
\newcommand{\setOfTransformskDims}{\setOfTransforms^{(k)}}
\newcommand{\bO}{\gO}
\newcommand{\dijNoisy}{d_{ij,\text{noisy}}}
\newcommand{\dijClean}{d_{ij,\text{clean}}}

\newcommand{\dataset}{\mathcal{X}}

\newcommand{\manifold}{\text{M}}

\newcommand{\relationship}{r}
\newcommand{\graphV}{\mathtt{V}}
\newcommand{\graphE}{\mathtt{E}}
\newcommand{\graphG}{\mathtt{G}}
\newcommand{\vdmS}{S}
\newcommand{\vdmD}{D}
\newcommand{\vdmC}{L(W,G)}
\newcommand{\vdmI}{\id}

\newcommand{\dmL}{L}
\newcommand{\dmI}{\id}

\newcommand{\pdf}{f}

\newcommand{\VectorFieldOnM}{X}
\newcommand{\VectorFieldOnMvec}{\bar{X}}

\newcommand{\embedding}{\iota}

\newcommand{\NN}{\mathbb{N}}

\newcommand{\RR}{\mathbb{R}}
\newcommand{\CC}{\mathbb{C}}
\newcommand{\EE}{\mathbb{E}}

\newcommand{\Ric}{\mbox{Ric}}

\newcommand{\ud}{\textup{d}}

\newtheorem{thm}{Theorem}[section]
\newtheorem{assumption}[thm]{Assumption}

\usepackage{mathtools}

\date{May 18th, 2014}
\title{Connection graph Laplacian methods can be made robust to noise}
\author{Noureddine El Karoui \thanks{UC Berkeley, Department of Statistics; \textbf{Contact:} nkaroui@berkeley.edu ;
Support from NSF grant DMS-0847647 (CAREER) is gratefully acknowledged.} and Hau-tieng Wu \thanks{Stanford University, Department of Mathematics; \textbf{Contact:} hauwu@stanford.edu ;
Support from AFOSR grant FA9550-09-1-0643 is gratefully acknowledged.\newline \textbf{Keywords: }  concentration of measure, random matrices, vector diffusion maps, spectral geometry, kernel methods;\newline \textbf{AMS MSC 2010 Classification:} 60F99, 53A99}}
\begin{document}
	\maketitle
\begin{abstract}
Recently, several data analytic techniques based on connection graph laplacian (CGL) ideas have appeared in the literature. At this point, the properties of these methods are starting to be understood in the setting where the data is observed without noise. 
We study the impact of additive noise on these methods, and show that they are remarkably robust. As a by-product of our analysis, we propose modifications of the standard algorithms that increase their robustness to noise. 
We illustrate our results in numerical simulations. 
\end{abstract}

\section{Introduction}
In the last few years, several interesting variants of kernel-based spectral methods have arisen in the applied mathematics literature. These ideas appeared in connection with new types of data, where pairs of objects or measurements of interest have a relationship that is ``blurred" by the action of a nuisance parameter. 
More specifically, we can find this type of data in a wide range of problems, for instance in the class averaging algorithm for the cryo-electron microscope (cryo-EM) problem \cite{singer_zhao_shkolnisky_hadani:2011,Zhao_Singer:2013}, 
in a modern light source imaging technique known as ptychography \cite{Marchesini_Tu_Wu:2014}, in graph realization problems \cite{Cucuringu_Lipman_Singer:2012,Cucuringu_Singer_Cowburn:2012}, in vectored PageRank \cite{Chung_Zhao_Kempton:2013}, in multi-channels image processing \cite{Batard_Sochen:2014}, etc...

Before we give further details about the cryo-EM problem, let us present the main building blocks of the methods we will study. They depend on the following three components:
\begin{enumerate}
\item an undirected graph $\graphG=(\graphV,\graphE)$ which describes all observations.  The observations are the vertices of the graph $\graphG$, denoted as $\{V_i\}_{i=1}^n$.\vspace{-.3cm}
\item an {\it affinity function} $w:\graphE\to \RR_+$, satisfying $w_{i,j}=w_{j,i}$, which describes how close two observations are ($i$ and $j$ index our observations). One common choice of $w_{i,j}=w(V_i,V_j)$ is of the form $w_{i,j}=\exp(-m(V_i,V_j)^2/\eps)$, where $m(x,y)$ is a metric measuring how far $x$ and $y$ are. \vspace{-.3cm}
\item a {\it connection function} $\relationship:\graphE\to \mathsf{G}$, where $\mathsf{G}$ is a Lie group, which describes how two samples are related. In its application to the cryo-EM problem, $r_{i,j}$'s can be thought of estimates of our nuisance parameters, which are orthogonal matrices.
\end{enumerate}

These three components form {\it the connection graph} associated with the data, which is denoted as $(\graphG, w,\relationship)$. They can be either given to the data analyst or have to be estimated from the data, depending on the application. 

This fact leads to different connection graph models and their associated noise models. For example, in the cryo-EM problem, all components of the connection graph $(\graphG, w,\relationship)$ are determined from the given projection images, where each vertex represents an image \cite[Appendix A]{ElKaroui_Wu:2013}; in the ptychography problem \cite{Marchesini_Tu_Wu:2014}, $\graphG$ is given by the experimenter, $\relationship$ is established from the experimental setup, and $w$ is built up from the diffractive images collected in the experiment. Depending on applications, different metrics, deformations or connections among pairs of observations are considered, or estimated from the dataset, to present the local information among data (see, for example, \cite{Batard_Sochen:2012,Boyer_Lipman_StClair_Puente_Patel_Funkhouser_Jernvall_Daubechies:2011,Al-Aifari_Daubechies_Lipman:2013,Collins_Zomorodian_Carlsson_Guibas:2004,Wang_Huang_Guibas:2013,Sun_Ovsjanikov_Guibas:2009,Talmon_Cohen_Gannot_Coifman:2013,Memoli_Sapiro:2005}). 

In this paper, since we focus on the connection graph Laplacian (CGL), we take the Lie Group $\mathsf{G}=O(k)$\footnote{We may also consider $U(k)$. But in this paper we focus on $O(k)$ to simplify the discussion.}, where $k\in\NN$, and assume that $\relationship$ satisfies $\relationship_{i,j}=\relationship^{-1}_{j,i}$. Our primary focus in this paper is on $k=1$ and $k=2$.

We now give more specifics about one of the problems motivating our investigation. 
\newline\newline
\textbf{Cryo-EM problem} In the cryo-EM problem, the experimenter collects 2-dimensional projection images of a 3-dimensional macro-molecular object of interest, and the goal is to reconstruct the 3-dimensional geometric structure of the macro-molecular object from these projection images. Mathematically, the collected images $\mathcal{X}_{\text{cryoEM}}:=\{I_i\}_{i=1}^N\in \RR^{m^2}$ can be modeled as the X-ray transform of the potential of the macro-molecular object of interest, denoted as $\psi:\RR^3\to \RR_+$. More precisely, in the setting that is usually studied, we have $I_i=X_\psi(R_i)$, where $R_i\in SO(3)$, $SO(3)$ is the 3-dimensional special orthogonal group, $X_\psi$ is the X-ray transform of $\psi$. The X-ray transform $X_\psi(R_i)$ is  a function from $\mathbb{R}^2$ to $\mathbb{R}_+$ and hence can be treated by the data analyst as an image. We refer the reader to \cite[Appendix A]{ElKaroui_Wu:2013} for precise mathematical details. (For the rest of the discussion, we write $R_i=[R_i^1\,\, R_i^2\,\, R_i^3]$ in the canonical basis, where $R_i^k$ are three dimensional unit vectors.) 

The experimental process produces data with high level of noise. Therefore, to solve this inverse problem, it is a common consensus to preprocess the images to increase the signal-to-noise ratio (SNR) before sending them to the cryo-EM data analytic pipeline. An efficient way to do so is to estimate the projection directions of these images, i.e $R_i^3$. This direction plays a particular role in the X-ray transform, which is different from the other two directions. If $R_i^3$'s were known, we would cluster the images according to these vectors and for instance take the mean of all properly rotationally aligned images to increase the SNR of the projection images - more on this below - in a cluster as a starting point for data-analysis. With these ``improved" images, we can proceed to estimate $R_i$ for the $i$-th image by applying, for example, the common line algorithm \cite{Hadani_Singer:2011b}, so that the 3-D image can be reconstructed by the inverse X-ray transform\cite{Epstein:2007}. We note that $R_i^3$ is a unit vector in $\mathbb{R}^3$ and hence lives on the standard sphere $S^2$. 

Conceptually, the problem is rendered difficult by the fact that the $X$-ray transform $X_\psi(R_i)$ is equivariant under the action of rotations that leave $R_i^3$ unchanged. In other words, if $r_\theta$ is an in-plane rotation, i.e a rotation that leaves $R_i^3$ unchanged but rotates $R_i^1$ and $R_i^2$ by an angle $\theta$, the image $X_\psi(r_\theta R_i)$ is $X_\psi(R_i)$ rotated by the angle $\theta$. In other words, $X_\psi(r_\theta R_i)=r_2(\theta)X_\psi(R_i)$, where $r_2(\theta)$ stands for the 2-dimensional rotation by the angle $\theta$. These in-plane rotations are clearly nuisance parameters if we want to evaluate the projection direction $R_i^3$. 

To measure the distance between $R_i^3$ and $R_j^3$, we hence use a rotationally invariant distance, i.e 
$d_{i,j}^2=\inf_{\theta \in [0,2\pi]}\norm{P_i-r_2(\theta)P_j}_2^2$. In other words, we look at the Euclidian distance between our two X-ray transforms/images after we have ``aligned" them as best as possible. We now think of $R^3_i$'s - the vectors we would like to estimate -  as elements of the manifold $S^2$, equipped with a metric $\mathsf{g}_\psi$, which depends on the macro-molecular object of interest. It turns out that the Vector Diffusion Maps algorithm (VDM), which is based on CGL and which we study in this paper, is effective in producing a good approximation of $\mathsf{g}_\psi$ from the local information $d_{i,j}$'s and the rotations we obtain by aligning the various X-ray transforms. This in turns imply better clustering of the $R_i^3$'s and improvement in the data-analytic pipeline for cryo-EM problems \cite{singer_zhao_shkolnisky_hadani:2011,Zhao_Singer:2013}.

The point of this paper is to understand how the CGL algorithms perform when the input data is corrupted by noise. The relationship between this method and the connection concept in differential geometry is the following: the projection images $P_i$ form a graph, and we can define the affinity and connection among a pair of images so that the topological structure of the 2-dimensional sphere $(S^2,\mathsf{g}_\psi)$ is encoded in the graph. This amounts to using the local geometric information derived from our data to estimate the global information - including topology - of $(S^2,\mathsf{g}_\psi)$.

\textbf{Impact of noise on these problems} What is missing from these considerations and the current literature is an understanding of how noise impact the procedures which are currently used and have mathematical backing in the noise-free context. The aim of our paper is to shed light on the issue of the impact of noise on these interesting and practically useful procedures. We will be concerned in this paper with the impact of adding noise on the observations - collected for instance in the way described above. 

Note that additive noise may have impact in all three building blocks of the connection graph associated with the data. First, it might make the graph noisy. For example, in the cryo-EM problem, the standard algorithm builds up the graph from a given noisy data set $\{P_i\}_{i=1}^n=\{I_i+\xi_i\}_{i=1}^n$ - $\xi_i$ is our additive noise - using the nearest neighbors determined by a pre-assigned metric, that is, we add an edge between two vertices when they are close enough in that metric. Then, clearly, the existence of the noise $\xi_i$ will likely create a different nearest neighbor graph from the the one that would be built up from the (clean) projection images $\{I_i\}_{i=1}^n$. As we will see in this paper, in some applications, it might be beneficial to consider a complete graph instead of a nearest neighbor graph. 

The second noise source is how $w$ and $\relationship$ are provided or determined from the samples. For example, in the cryo-EM problem, although $\{P_i\}$ are points located in a high dimensional Euclidean space, we determine the affinity and connection between two images by evaluating their rotationally invariant distance. It is clear that when $P_i$ is noisy, the derived affinity and connection will be noisy and likely quite different from the affinity and connection we would compute from the clean dataset $\{I_i\}_{i=1}^n$. On the other hand, in the ptychography problem, the connection is directly determined from the experimental setup, so that it is a noise-free even when our observations are corrupted by additive noise. 

In summary, corrupting the observations by additive noise might impact the following elements of the data analysis:
\begin{enumerate}
\item which scheme and metric we choose to construct the graph; \vspace{-.3cm}
\item how we build up the affinity function; \vspace{-.3cm}
\item how we build up the connection function.
\end{enumerate}

\textbf{More details on CGL methods}\\
At a high-level, connection graph Laplacian (CGL) methods create a block matrix from the connection graph.  The spectral properties of this matrix are then used to estimate properties of the intrinsic structure from which we posit the data is drawn from. This in turns lead to good estimation methods for, for instance, geodesic distance on the manifold, if the underlying intrinsic structure is a manifold. We refer the reader to Appendix \ref{Appendix:Section:CGL} and references \cite{singer_wu:2012,singer_wu:2013,Bandeira_Singer_Spielman:2013,Chung_Kempton:2013,Chung_Zhao_Kempton:2013} for more information.

Given a $n\times n$ matrix $W$, with scalar entries denoted by $w_{i,j}$ and a ${nk\times nk}$ block matrix $G$ with $k\times k$ block entries denoted by $G_{i,j}$, we define a $nk\times nk$ matrix $S$ with $(i,j)$-block entries 
$$\vdmS_{i,j}=w_{i,j} G_{i,j}$$
and a $nk\times nk$ block diagonal matrix $D$ with $(i,i)$-block entries 
$$\vdmD_{i,i}=\sum_{j\neq i} w_{i,j} \id_k,$$
which is assumed to be invertible. Let us call 
\begin{align}\label{definition:LWG}
L(W,G):=\vdmD^{-1} \vdmS\,\,\,\mbox{ and }\,\,
L_0(W,G):=L(W\circ 1_{i\neq j}, G). 
\end{align}
In other words, $L_0(W,G)$ is the matrix $L(W,G)$ computed from the weight matrix $W$ where the diagonal weights have been replaced by $0$.

Suppose we are given a connection graph $(\graphG,w,\relationship)$, and construct the $n\times n$ {\it affinity matrix} $W$ so that $w_{i,j}=w(i,j)$ and the {\it connection matrix} $G$, the ${nk\times nk}$ block matrix with $k\times k$ block entries $G_{i,j}=\relationship(i,j)$, {\it the CGL associated with the connection graph $(\graphG,w,\relationship)$} is defined as $\vdmI_{nk}-L(W,G)$ and {\it the modified CGL associated with the connection graph $(\graphG,w,\relationship)$} is defined as $\vdmI_{nk}-L_0(W,G)$. We note that under our assumptions on $\relationship$, i.e $r_{i,j}=r_{j,i}^{-1}=r_{i,j}^*$ the connection matrix $G$ is Hermitian.

We are interested in the large eigenvalues of $L(W,G)$ (or, equivalently, the small eigenvalues of the CGL $\vdmI_{nk}-L(W,G)$), as well as the corresponding eigenvectors. In the case where the data is not corrupted by noise, the CGL's asymptotic properties have been studied in \cite{singer_wu:2012,singer_wu:2013}, when the underlying intrinsic structure is a manifold. Its so-called synchronization properties have been studied in \cite{Bandeira_Singer_Spielman:2013,Chung_Zhao_Kempton:2013}.

The aim of our study is to understand the impact of additive noise on CGL algorithms. Two main results are Proposition \ref{prop:controlApproxdijNoisy}, which explains the effect of noise on the affinity, and Theorem \ref{thm:consistencyRotations}, which explains the effect of noise on the connection. These two results lead to suggestions for modifying the standard CGL algorithms: the methods are more robust when we use a complete graph than when we use a nearest-neighbor graph, the latter being commonly used in practice. One should also use the matrix $L_0(W,G)$ instead of $\vdmC$ to make the method  more robust to noise. After we suggest these modifications, our main result is Proposition \ref{prop:connectionGraphLaplacianApproxModifS}, which shows that even when the signal-to-noise-ratio (SNR) is very small, i.e going to 0 asymptotically, our modifications of the standard algorithm will approximately yield the same spectral results as if we had been working on the CGL matrix computed from noiseless data. 
We develop in Section 2 a theory for the impact of noise on CGL algorithms and show that our proposed modifications to the standard algorithms render them more robust to noise. We present in Section 3 some numerical results.

\textbf{Notation:} Here is a set of notations we use repeatedly. ${\cal T}$ is a set of linear transforms. $\id_k$ stands for the $k\times k$ identity matrix. If $v\in\mathbb{R}^n$, $D(\{v\})$ is a $nk\times nk$ block diagonal matrix with the $(i,i)$-th block equal to $v_i \id_k$. We denote by $A\circ B$ the Hadamard, i.e entry-wise, product of the matrices $A$ and $B$. $\opnorm{M}$ is the largest singular value (a.k.a operator norm) of the matrix $M$. $\norm{M}_F$ is its Frobenius norm.

\section{Theory}
Our aim in this section is to develop a theory that explains the behavior of CGL algorithms in the presence of noise. In particular, it will apply to  algorithms of the cryo-EM type. We give in Subsection \ref{subsec:generalApproxResMatrices} approximation results that apply to general CGL problems. In Subsection \ref{subsec:impactNoiseVDMProcedure}, we study in details the impact of noise on both the affinity and the connection used in the computation of the CGL when using the rotationally invariant distance (this is particularly relevant for the cryo-EM problem). We put these results together for a detailed study of CGL algorithms in Subsection \ref{subsec:consequencesForVDM}. We also propose in Subsection \ref{subsec:consequencesForVDM} modifications to the standard algorithms. 

\subsection{General approximation results}\label{subsec:generalApproxResMatrices}
We first present a result that applies generally to CGL algorithms. 
\begin{lemma}\label{lemma:approxBoundsLaplacian}
Suppose $W$ and $\widetilde{W}$ are $n\times n$ matrices, with scalar entries denoted by $w_{i,j}$ and $\tilde{w}_{i,j}$ and $G$ and $\widetilde{G}$ are $nd\times nd$ block matrices, with $d\times d$ blocks denoted by $G_{i,j}$ and $\tilde{G}_{i,j}$. 
Suppose that 
$$
\sup_{i,j} |\tilde{w}_{i,j}-w_{i,j}|\leq \eps \;, \text{ and } \sup_{i,j} \norm{\widetilde{G}_{i,j}-G_{i,j}}_F \leq \eta\;,
$$
where $\eps,\eta\geq 0$. Suppose furthermore that there exists $C>0$ such that $0\leq w_{i,j}\leq C$, 
$\sup_{i,j} \norm{G_{i,j}}_F\leq C$ and $\sup_{i,j} \norm{\widetilde{G}_{i,j}}_F\leq C$. 
Then, if $\inf_i \sum_{j\neq i} w_{i,j}/n>\gamma$ and $\gamma>\eps$, 
we have 
$$
\opnorm{L(W,G)-L(\widetilde{W},\widetilde{G})}\leq \frac{1}{\gamma} C(\eta+\eps)+\frac{\eps}{\gamma (\gamma-\eps)} C^2\;.
$$

\end{lemma}

The proof of this lemma is given in Appendix \ref{app:sec:ApproxResLaplMatrix}. 
This lemma says that if we can approximate the matrix $W$ well entrywise and each of the individual matrices $G_{i,j}$ well, too, data analytic techniques working on the CGL matrix $L(\widetilde{W},\widetilde{G})$ will do essentially as well as those working on the corresponding matrix for $L(W,G)$ in the spectral sense. 

This result is useful because many 
methods rely on these connection graph ideas, with different input in terms of affinity and connection functions \cite{singer_zhao_shkolnisky_hadani:2011,Zhao_Singer:2013,Cucuringu_Lipman_Singer:2012,Cucuringu_Singer_Cowburn:2012,Boyer_Lipman_StClair_Puente_Patel_Funkhouser_Jernvall_Daubechies:2011,Al-Aifari_Daubechies_Lipman:2013,Chen_Lin_Chern:2013,Collins_Zomorodian_Carlsson_Guibas:2004,Sun_Ovsjanikov_Guibas:2009,Talmon_Cohen_Gannot_Coifman:2013,Memoli_Sapiro:2005}. 
However, it will often be the case that we can approximate $w_{i,j}$ - which we think of as measurements we would get if our signals were not corrupted by noise - only up to a constant. The following result shows that in certain situations, this will not affect dramatically the spectral properties of the CGL matrix.

\begin{lemma}\label{lemma:approxBoundsLaplacianMultiplicativeErrorAffinity}
We work under the same setup as in Lemma \ref{lemma:approxBoundsLaplacian} and with the same notations. 
However, we now assume that 
$$
\exists \{f_i\}_{i=1}^n\;, f_i>0 : \sup_{i,j} \left|\frac{\tilde{w}_{i,j}}{f_i}-w_{i,j}\right|\leq \eps \;, \text{ and } \sup_{i,j} \norm{\widetilde{G}_{i,j}-G_{i,j}}_F \leq \eta\;.
$$
Suppose furthermore that there exists $C>0$ such that $0\leq w_{i,j}\leq C$,  
$\sup_{i,j} \norm{G_{i,j}}_F\leq C$ and $\sup_{i,j} \norm{\widetilde{G}_{i,j}}_F\leq C$. 
Then, if $\inf_i \sum_{j\neq i} w_{i,j}/n>\gamma$ and $\gamma>\eps$, 
we have 
$$
\opnorm{L(W,G)-L(\widetilde{W},\widetilde{G})}\leq \frac{1}{\gamma} C(\eta+\eps)+\frac{\eps}{\gamma (\gamma-\eps)} C^2\;.
$$
\end{lemma}
We note that quite remarkably, there are essentially no conditions on $f_i$'s: in particular, $\widetilde{w}_{i,j}$ and $w_{i,j}$ could be of completely different magnitudes. The previous lemma also shows that, for the purpose of understanding the large eigenvalues and eigenvectors of $L(W,G)$, we do not need to estimate $f_i$'s: we can simply use $L(\widetilde{W},\widetilde{G})$, i.e just work with the noisy data. 

\begin{proof}
Let us call $\widetilde{W_f}$ the matrix with scalar entries $\tilde{w}_{i,j}/f_i$. We note simply that 
$$
L(\widetilde{W_f},\widetilde{G})=L(\widetilde{W},\widetilde{G})\;.
$$
The assumptions of Lemma \ref{lemma:approxBoundsLaplacian} apply to $(\widetilde{W_f},\widetilde{G})$ and hence we have 
$$
\opnorm{L(W,G)-L(\widetilde{W_f},\widetilde{G})}\leq \frac{1}{\gamma} C(\eta+\eps)+\frac{\eps}{\gamma (\gamma-\eps)} C^2\;.
$$
But since $L(\widetilde{W_f},\widetilde{G})=L(\widetilde{W},\widetilde{G})$, we also have 
$$
\opnorm{L(W,G)-L(\widetilde{W},\widetilde{G})}\leq \frac{1}{\gamma} C(\eta+\eps)+\frac{\eps}{\gamma (\gamma-\eps)} C^2\;.
$$
\end{proof}

In some situations that will be of interest to us below, it is however, not the case that we can find $f_i$'s such that 
$$
\exists \{f_i\}_{i=1}^n\;, f_i>0 :\, \sup_{i,j} \left|\frac{\tilde{w}_{i,j}}{f_i}-w_{i,j}\right|\leq \eps\;.
$$
Rather, this approximation is possible only when $i\neq j$, yielding the condition 
$$
\forall i, \exists f_i>0 :\, \sup_{i\neq j} \left|\frac{\tilde{w}_{i,j}}{f_i}-w_{i,j}\right|\leq \eps\;.
$$

This apparently minor difference turns out to have significant consequences, both practical and theoretical. We propose in the following lemma to modify the standard way of the computing the CGL matrix to handle this more general case.  

\begin{lemma}\label{lemma:approxCGLMatGeneralModif}
We work under the same setup as in Lemma \ref{lemma:approxBoundsLaplacian} and with the same notations. 
We now assume that multiplicative approximations of the weights is possible only on the off-diagonal elements of our weight matrix:
$$
\exists \{f_i\}_{i=1}^n\;, f_i>0 : \sup_{i\neq j} \left|\frac{\tilde{w}_{i,j}}{f_i}-w_{i,j}\right|\leq \eps \;, \text{ and } \sup_{i,j} \norm{\widetilde{G}_{i,j}-G_{i,j}}_F \leq \eta\;.
$$
Suppose furthermore that there exists $C>0$ such that $0\leq w_{i,j}\leq C$,  
$\sup_{i,j} \norm{G_{i,j}}_F\leq C$, and $\sup_{i,j} \norm{\widetilde{G}_{i,j}}_F\leq C$. 
Then, if $\inf_i \sum_{j\neq i} w_{i,j}/n>\gamma$ and $\gamma>\eps$, 
we have 
$$
\opnorm{L_0(W,G)-L_0(\widetilde{W},\widetilde{G})}\leq \frac{1}{\gamma} C(\eta+\eps)+\frac{\eps}{\gamma (\gamma-\eps)} C^2\;,
$$
and 
$$
\opnorm{L(W,G)-L_0(\widetilde{W},\widetilde{G})}\leq 
\frac{1}{\gamma} C(\eta+\eps)+\frac{\eps}{\gamma (\gamma-\eps)} C^2+\frac{C^2}{n\gamma }\;.
$$
\end{lemma}

\textbf{Comment:} Concretely, this lemma means that if we do not include the block diagonal terms in the computation of the CGL obtained from our ``noisy data'', i.e $(\widetilde{W},\widetilde{G})$, we will get a matrix that is very close in spectral norm to the CGL computed from the ``clean data'', i.e $(W,G)$. The significance of this result lies in the fact that recent work in applied mathematics has proposed to use the large eigenvalues and eigenvectors of $L(W,G)$ for various data analytic tasks, such as the estimation of local geodesic distances when the data is thought to be sampled from an unknown manifold. 

What our result shows is that even when $f_i$ are arbitrarily large, which we can think of as the situation where the signal to noise ratio in $\widetilde{W}$ is basically 0, working with $L_0(\widetilde{W},\widetilde{G})$ will allow us to harness the power of these recently developed tools. Naturally, working with $(\widetilde{W},\widetilde{G})$ is a much more realistic assumption than working with $(W,G)$ since we expect all our measurements to be somewhat noisy whereas results based on $(W,G)$ essentially assume that there is no measurement error in the dataset. 

\begin{proof}
Recall also that the computation of the $D$ matrix does not involve the diagonal weights. Therefore, 
$$
L(W,G)=L_0(W,G)+D^{-1} \Delta(\{w_{i,i}\},G_{i,i})\;,
$$
where $\Delta(\{w_{i,i}\},G_{i,i})$ is the block diagonal matrix with $(i,i)$ block diagonal $w_{i,i}G_{i,i}$, and $D^{-1} \Delta(\{w_{i,i}\},G_{i,i})$ is a block diagonal matrix with $(i,i)$ block diagonal 
$$
\frac{w_{i,i}}{\sum_{j\neq i} w_{i,j}} G_{i,i}\;.
$$
This implies that 
$$
\opnorm{L(W,G)-L_0(W,G)}\leq \sup_i \frac{w_{i,i}}{\sum_{j\neq i} w_{i,j}} \opnorm{G_{i,i}}\;.
$$
Our assumptions imply that $\sum_{j\neq i} w_{i,j}>\gamma n$, $\sup_{i}w_{i,i}\leq C$ and $\opnorm{G_{i,i}}\leq \norm{G_{i,i}}_F\leq C$. Hence, 
$$
\opnorm{L(W,G)-L_0(W,G)}\leq \frac{C^2}{n\gamma}\;.
$$
We still have $L(\widetilde{W_f},\widetilde{G})=L(\widetilde{W},\widetilde{G})$	
and of course
$$
L_0(\widetilde{W_f},\widetilde{G})=L_0(\widetilde{W},\widetilde{G})\;.
$$

Note further that the assumptions of Lemma \ref{lemma:approxBoundsLaplacianMultiplicativeErrorAffinity} apply to the matrices $(W\circ 1_{i\neq j}, G)$ and $(\widetilde{W}\circ 1_{i\neq j},\widetilde{G})$. Indeed, the off-diagonal conditions on the weights are the assumptions we are working under. The diagonal conditions on the weights in Lemma \ref{lemma:approxBoundsLaplacianMultiplicativeErrorAffinity} are trivially satisfied here since both $W\circ 1_{i\neq j}$ and $\widetilde{W}\circ 1_{i\neq j}$ have diagonal entries equal to 0. Hence, Lemma \ref{lemma:approxBoundsLaplacianMultiplicativeErrorAffinity} gives 
$$
\opnorm{L_0(W,G)-L_0(\widetilde{W},\widetilde{G})}\leq \frac{1}{\gamma} C(\eta+\eps)+\frac{\eps}{\gamma (\gamma-\eps)} C^2\;.
$$
Using Weyl's inequality (see \cite{bhatia97}) and our bound on $\opnorm{L(W,G)-L_0(W,G)}$, we therefore get 
$$
\opnorm{L(W,G)-L_0(\widetilde{W},\widetilde{G})}\leq \frac{1}{\gamma} C(\eta+\eps)+\frac{\eps}{\gamma (\gamma-\eps)} C^2+\frac{C^2}{n\gamma }\;.
$$

\end{proof}

We now turn to the analysis of a specific algorithm, the class averaging algorithm in the cryo-EM problem, with a broadly accepted model of noise contamination to demonstrate the robustness of CGL-like algorithms.

\subsection{Impact of noise on the rotationally invariant distance}
\label{subsec:impactNoiseVDMProcedure}
We assume that we observe noisy versions of  the $k$-dimensional images/objects, $k\geq 2$, we are interested in. If the images in the - unobserved - clean dataset are called $\{\cleanImage_i\}_{i=1}^n$, we observe 
$$
\noisyImage_i=\cleanImage_i+ N_i\;.
$$
Here $\{N_i\}_{i=1}^n$ are pure-noise images/objects. 
Naturally, after discretization, the images/objects we consider are just data vectors of dimension $\numberPixels$ -- we view $\cleanImage_i$ and $N_i$ as vectors in $\mathbb{R}^\numberPixels$. In other words, for a $k$-dim image, we sample $p$ points from the domain $\RR^k$, which is denoted as $\mathfrak{X}:=\{\mathsf{x}_i\}_{i=1}^p\subset \RR^k$ and called the {\it sampling grid}, and the image is discretized accordingly on these points.
We also assume that the random variables $N_i$'s, $i=1,\ldots,n$, are independent. 

\subsubsection{Distance measurement between pairs of images}
We start from a general definition. Take a set of linear transforms $\setOfTransformskDims\subset O(k)$. Consider the following measurement between two objects/images, $d_{ij}\geq 0$, with  
$$
d_{ij}^2=\inf_{\mathtt{O}\in \setOfTransformskDims }\norm{\noisyImage_i-\mathtt{O}\circ\noisyImage_j}^2_2\;,
$$
where $\circ$ means that the transform is acting on the pixels. For example, in the continuous setup where $I_j$ is replaced by $f_j \in L^2(\RR^k)$, given $\mathtt{O}\in SO(k)$, we have $\mathtt{O}\circ f_j(x):=f_j(\mathtt{O}^{-1} x)$ for all $x\in\RR^k$. When $\setOfTransformskDims=SO(k)$, $d_{ij}$ is called {\it the rotationally invariant distance (RID)}. 

In the discrete setup of interest in this paper, we assume that $\mathfrak{X}=\mathtt{O}^{-1}\mathfrak{X}$ for all $\mathtt{O}\in \setOfTransformskDims$; that is, the linear transform is {\it exact} (with respect to the grid $\mathfrak{X}$), in that it maps the sampling grid onto itself. For concreteness, here is an example of sampling grid and associated exact linear transforms. Let $k=2$ and take the sampling grid to be the polar coordinates grid. Since we are in dimension 2, we pick $m$ rays of length $1$ at angles $2\pi k/m$, $k=0,\ldots,m-1$ and have $l$ equally spaced points on each of those rays. We consider $I_i$ to be the discretization of the function $f_i \in L^2(\mathbb{R}^2)$ which is compactly supported inside the unit disk, at the polar coordinate grid. The set $\setOfTransforms^{(2)}$ consisting of elements of $SO(2)$ with angles $\theta_k=2\pi\frac{k}{m}$, where $k=1,\ldots,m$, is thus exact and associated to the polar coordinate grid.

The discretization and notation merit further discussion. As a linear transform of the domain $\RR^k$, $\mathtt{O}\in \setOfTransformskDims$ can be represented by a $k\times k$ matrix. On the other hand, in the discretized setup we consider here, we can map $\setOfTransformskDims$ to a set $\setOfTransforms$ of $p\times p$ matrices $\transform$ which acts on the discretized images $\noisyImage_j$. These images are viewed as a set of $p$-dim vectors, denoted as $\noisyImage_j^\vee$, and $\transform$ acts on a ``flattened'' or ``vectorized'' (i.e 1-dimensional) version of the $k$-dimensional object of interest. Note that to each transform $\mathtt{O}$ there corresponds a unique $p\times p$ matrix $\transform$. In the following, we will use $\mathtt{O}$ to denote the transform acting on the pixels, and use $\transform$ to mean its companion matrix acting on the vectorized version of the object we are interested in. A simple but very important observation is that 
$$
(\mathtt{O}\circ I_i)^\vee=O I_i^\vee\;.
$$
In other words, we will have $\inf_{\mathtt{O}\in \setOfTransformskDims }\norm{\noisyImage_i-\mathtt{O}\circ\noisyImage_j}=\inf_{\transform\in \setOfTransforms }\norm{\noisyImage^\vee_i-\transform\noisyImage^\vee_j}$. To simplify the notation, when it is clear from the context, we will use $\noisyImage_j$ to mean both the discretized object of interest and its vectorized version.

In what follows, we assume that ${\setOfTransforms}$ always contains $\id_p$. We study the impact of noise on $d_{ij}$ through a uniform approximation argument. 
Let us call for $\transform \in \setOfTransforms$,
\begin{align*}
\dijNoisy^2(\transform):=\norm{\noisyImage_i^\vee-\transform \noisyImage_j^\vee}^2 \;, \text{ and }
\dijClean^2(\transform):=\norm{\cleanImage_i^\vee-\transform \cleanImage_j^\vee}^2\;.
\end{align*}
Essentially we will show that, when $\setOfTransforms$ contains only orthogonal matrices and is not ``too large", 
$$
\sup_{\transform \in \setOfTransforms }\sup_{i\neq j} |\dijNoisy^2(\transform)-\dijClean^2(\transform)-f(i,j)|=\lo_P(1)\;,
$$
where $f(i,j)$ does not depend on $\transform$.
Our approximations will in fact be much more precise than this. But we will be able to conclude that in these circumstances, 
$$
\sup_{i\neq j}\left|\inf_{\transform \in \setOfTransforms }\dijNoisy^2(\transform)-\inf_{\transform \in \setOfTransforms }\dijClean^2(\transform)-f(i,j)\right|=\lo_P(1)\;.
$$

We have the following theorem for any given set of transforms $\setOfTransforms$. 

\begin{theorem}\label{thm:controlApproxUnifIndicesAndTransfo}
Suppose that for $1\leq i\leq n$, $N_i$ are independent, with $N_i^\vee\sim {\cal N}(0,\Sigma_i)$. 
Call $\mathsf{t_p}:=\sup_i \sup_{\transform \in \setOfTransforms }\sqrt{\trace{(\transform\Sigma_i\transform\trsp)^2}}$ and $\mathsf{s}_p:=\sup_{1\leq i \leq n}\sup_{\transform \in \setOfTransforms }\sqrt{\opnorm{\transform\Sigma_i\transform\trsp}}$. 
Then, we have 
\begin{align*}
&\sup_{\transform \in \setOfTransforms }\sup_{i\neq j} |\dijNoisy^2(\transform)-\dijClean^2(\transform)-\trace{\Sigma_i+\transform\Sigma_j\transform\trsp}| \notag\\
&=\bO_P\left(
\sqrt{\log [\card{\setOfTransforms} n^2]} \Big(\mathsf{t_p}+\mathsf{s}_p \sup_{i,\transform \in \setOfTransforms} \norm{\transform \cleanImage_i\transform\trsp}\Big)+\log [\card{\setOfTransforms} n^2]\mathsf{s}_p^2\right)\;.
\end{align*}
\end{theorem}

\begin{proof}
We first note that $N_i^\vee-\transform N_j^\vee\sim {\cal N}(0,\Sigma_i+\transform \Sigma_j \transform\trsp)$. 
Applying Lemma \ref{lemma:controlSupQuadFormsLaurentMassart} to $\norm{N_i^\vee-\transform N_j^\vee}^2$ with $Q_i=\id$, we get 
\begin{align*}
&\sup_{\transform\in\setOfTransforms} \sup_{i\neq j} \left|\norm{N_i^\vee-\transform N_j^\vee}^2-\trace{\Sigma_i+\transform\Sigma_j\transform\trsp}\right|=\\
&\bO_P\left(\sqrt{\log [\card{\setOfTransforms} n^2]} 
\sup_{i,j,\transform} \left[\sqrt{\trace{(\Sigma_i+\transform \Sigma_j \transform\trsp)^2}}\right]+\log [\card{\setOfTransforms} n^2] \sup_{i,j,\transform} \opnorm{(\Sigma_i+\transform \Sigma_j \transform\trsp)}\right)\;.
\end{align*}
Of course, using the fact that for positive semi-definite matrices, $(A+B)^2\preceq 2(A^2+B^2)$ in the positive-semidefinite order, we have
$$
\trace{(\Sigma_i+\transform \Sigma_j \transform\trsp)^2}\leq 2\trace{\Sigma_i^2+[\transform\Sigma_j\transform\trsp]^2} \text\;.
$$
Hence,
\begin{align*}
&\sup_{\transform\in\setOfTransforms} \sup_{i\neq j} \left|\norm{N_i^\vee-\transform N_j^\vee}^2-\trace{\Sigma_i+\transform\Sigma_j\transform\trsp}\right|=\\
&\quad\bO_P\left(\sqrt{\log [\card{\setOfTransforms} n^2]} 
\sup_{i} \sup_{\transform\in\setOfTransforms}\sqrt{\trace{\left[\transform\Sigma_i\transform\trsp\right]^2}}+\log [\card{\setOfTransforms} n^2] \sup_i\sup_{\transform\in\setOfTransforms} \opnorm{\transform\Sigma_i\transform\trsp}\right)\;.
\end{align*}
We also note that 
$$
(\cleanImage_i^\vee-\transform \cleanImage_j^\vee)\trsp (N_i^\vee-\transform N_j^\vee)\sim {\cal N}(0,\gamma_{i,j,\transform}^2)\;,
$$
where 
$$
\gamma_{i,j,\transform}^2=(\cleanImage_i^\vee-\transform \cleanImage_j^\vee)\trsp (\Sigma_i-\transform\Sigma_j \transform\trsp)(\cleanImage_i^\vee-\transform \cleanImage_j^\vee)\;.
$$
We note that 
$$
\gamma_{i,j,\transform}^2\leq 2\left(\norm{\cleanImage_i^\vee}^2+\norm{\transform\cleanImage_j^\vee\transform\trsp}^2\right)(\opnorm{\Sigma_i}+\opnorm{\transform\Sigma_j\transform\trsp})\leq 8\sup_{i,\transform \in \setOfTransforms} \opnorm{\transform\Sigma_i\transform\trsp} \sup_{i,\transform \in \setOfTransforms} \norm{\transform\cleanImage_i^\vee\transform\trsp}^2\;.
$$
Recall also that it is well known that if $Z_1,\ldots, Z_N$ are ${\cal N}(0,\gamma_i^2)$ random variables, 
$$
\sup_{1\leq k \leq N} |Z_k|=\bO_P(\sqrt{\log N} \sup_k \gamma_k)\;.
$$
This result can be obtained by a simple union bound argument.
In our case, it means that 
$$
\sup_{\transform \in \setOfTransforms} \sup_{i\neq j}|(\cleanImage_i^\vee-\transform \cleanImage_j^\vee)\trsp (N_i^\vee-\transform N_j^\vee)|=\bO_P\left(
\sqrt{\log [\card{\setOfTransforms} n^2]} \sup_{i,\transform \in \setOfTransforms} \sqrt{\opnorm{\transform\Sigma_i\transform\trsp}} \sup_{i,\transform \in \setOfTransforms} \norm{\transform\cleanImage_i^\vee\transform\trsp}
\right)\;.
$$
\end{proof}

In light of the previous theorem, we have the following proposition.  
\begin{proposition}\label{prop:controlApproxdijNoisy}
Suppose that for all $1\leq i \leq n$ and $\transform \in \setOfTransforms$, $\opnorm{\transform\Sigma_i\transform\trsp}\leq \sigma^2_p$, $\sqrt{\trace{[\transform\Sigma_i\transform\trsp]^2}/p}\leq s_p^2$, and $\norm{\transform\cleanImage_i^\vee\transform\trsp}\leq K$, where $K$ is a constant independent of $p$. Then, 
\begin{align}
\sup_{\transform \in \setOfTransforms }\sup_{i\neq j} |\dijNoisy^2(\transform)-\dijClean^2(\transform)-\trace{\Sigma_i+\transform\Sigma_j\transform\trsp}|=\bO_P(u_{n,p})\;.\label{definition:unp}
\end{align}
where $u_{n,p}:=\sqrt{\log [\card{\setOfTransforms} n^2]} (\sqrt{p}s^2_p+K\sigma_p)+\log [\card{\setOfTransforms} n^2]\sigma_p^2$. \\
It follows that, if $\sqrt{\log [\card{\setOfTransforms} n^2]} \max(\sqrt{p}s^2_p,\sigma_p)\tendsto 0$, and $\setOfTransforms$ contains only orthogonal matrices, 
\begin{align*}
\sup_{\transform \in \setOfTransforms }\sup_{i\neq j} |\dijNoisy^2(\transform)-\dijClean^2(\transform)-\trace{\Sigma_i+\Sigma_j}|=\bO_P(u_{n,p})\;=\lo_P(1)\;.
\end{align*}

Furthermore, in this case, 
$$
\sup_{i\neq j}\left|d_{ij,\text{noisy}}^2-d_{ij,\text{clean}}^2-\trace{\Sigma_i+\Sigma_j}\right|=\lo_P(1)\;,
$$
where
\begin{align*}
d_{ij,\text{noisy}}^2:=\inf_{\transform\in\setOfTransforms} \norm{\noisyImage_i^\vee-\transform \noisyImage_j^\vee}^2\;,\quad
d_{ij,\text{clean}}^2:=\inf_{\transform\in\setOfTransforms} \norm{\cleanImage_i^\vee-\transform \cleanImage_j^\vee}^2\;.
\end{align*}
\end{proposition}
In light of the previous proposition, the following set of assumptions is natural:
\paragraph{Assumption G1 : }$\forall i, \transform \in \setOfTransforms$, $\opnorm{\transform\Sigma_i\transform\trsp}\leq \sigma^2_p$, $\sqrt{\trace{[\transform\Sigma_i\transform\trsp]^2}/p}\leq s_p^2$, and $\norm{\transform\cleanImage_i^\vee\transform\trsp}\leq K$, where $K$ is a constant independent of $p$. Furthermore, $\sqrt{\log [\card{\setOfTransforms} n^2]} \max(\sqrt{p}s^2_p,\sigma_p)\tendsto 0$ and hence $u_{n,p}\tendsto 0$.

We refer the reader to Proposition \ref{prop:cardinalSetOfDiscretizedRotations} on page \pageref{prop:cardinalSetOfDiscretizedRotations} for a bound on $\card{\setOfTransforms}$ that is relevant to the class averaging algorithm in the cryo-EM problem. 

\begin{proof}[Proof of Proposition \ref{prop:controlApproxdijNoisy}]
The first two statements are immediate consequences of Theorem \ref{thm:controlApproxUnifIndicesAndTransfo}. For the second one, we use the fact that 
since $\transform \in \setOfTransforms$ is orthogonal, $\trace{\transform \Sigma_j\transform \trsp}=\trace{\Sigma_j}$.

Now, if $F$ and $G$ are two functions, we clearly have $|\inf F(x)-\inf G(x)|\leq \sup |G(x)-F(x)|$. Indeed, $\forall x$, $F(x)\leq G(x)+\sup |G(x)-F(x)|$. Hence, for all $x$, 
$$
\inf_x F(x)\leq F(x)\leq  G(x)+\sup |G(x)-F(x)|\;,
$$
and we conclude by taking $\inf$ in the right-hand side. The inequality is proved similarly in the other direction. 
The results of Theorem \ref{thm:controlApproxUnifIndicesAndTransfo} therefore show that 
$$
\sup_{i\neq j}\left|d_{ij,\text{noisy}}^2-d_{ij,\text{clean}}^2-\trace{\Sigma_i+\Sigma_j}\right|=\bO_P(\sqrt{\log [\card{\setOfTransforms} n^2]} (\sqrt{p}s^2_p+K\sigma_p)+\log [\card{\setOfTransforms} n^2]\sigma_p^2)
$$
and we get the announced conclusions under our assumptions. 
\end{proof}

We now present two examples to show that our assumptions are quite weak and prove that the algorithms we are studying can tolerate very large amount of noise.

\paragraph{Magnitude of noise: First example}  Assume that $N_i^\vee\sim p^{-(1/4+\eps)}{\cal N}(0,\id_p)$, where $\eps>0$. In this case, $\norm{N_i^\vee}\sim p^{1/4-\eps}\gg \sup_i \norm{S_i^\vee}$ if $\eps<1/4$. In other words, the norm of the error vector is much larger than the norm of the signal vector. Indeed, asymptotically, the signal to noise ratio $\norm{S_i^\vee}/\norm{N_i^\vee}$ is 0. Furthermore, $\sigma_p=p^{-(1/4+\eps)}$ and $\sqrt{p}s_p^2=p^{-2\eps}$. Hence, if $\card{\setOfTransforms}=\bO(p^{\gamma})$ for some $\gamma$, our conditions translate into $\sqrt{\log(np)} \max(p^{-(1/4+\eps)},p^{-2\eps})\tendsto 0$. This is of course satisfied provided $n$ is subexponential in $p$. See Proposition \ref{prop:cardinalSetOfDiscretizedRotations} for a natural example of $\setOfTransforms$ whose cardinal is polynomial in $p$. 

\paragraph{Magnitude of noise: Second example} We now consider the case where $\Sigma_i$ has one eigenvalue equal to $p^{-\eps}$ and all the others are equal to $p^{-(1/2+\eta)}$, $\eps, \eta>0$. In other words, the noise is much larger in one direction than in all the others. In this case, $\sigma_p^2=p^{-\eps}$ and $\trace{\Sigma_i^2}=p^{-2\eps}+(p-1)*p^{-(1+2\eta)}\leq p^{-2\eps}+p^{-2\eta}$. So if once again, $\card{\setOfTransforms}=\bO(p^\gamma)$, our conditions translate into $\sqrt{\log(np)} \max(p^{-\eps}+p^{-\eta},p^{-\eps/2})\tendsto 0$. This example would also work if the number of eigenvalues equal to $p^{-\eps}$ were $\lo(p^{2\eps}/[\log (np)])$, provided $\sqrt{\log(np)}\max(p^{-\eta},p^{-\eps/2})\tendsto 0$.

\paragraph{Comment on the conditions on the signal in Assumption G1} At first glance, it might look like the condition $\sup_{i,\transform \in \setOfTransforms}\norm{\transform S_i^\vee \transform\trsp}\leq K$ is very restrictive due to the fact that, after discretization, $S_i$ has $p$ pixels. However, it is typically the case that if we start from a function in $L^2(\mathbb{R}^k)$, the discretized and vectorized image $S_i^\vee$ is normalized by the number of pixels $p$, so that $\norm{S_i^\vee}$ is roughly equal to the $L^2$-norm of the corresponding function. Hence, our condition  $\sup_{i,\transform \in \setOfTransforms}\norm{\transform S_i^\vee \transform\trsp}\leq K$ is very reasonable. 

\subsubsection{The case of ``exact rotations''}
We now focus on the most interesting case for our problem, namely the situation where $\mathtt{O}$ leaves our sampling grid invariant. We call $\setOfTransformskDims_{\text{exact}} \subset SO(k)$ the corresponding matrices $\mathtt{O}$ and $\setOfTransforms_{\text{exact}}$ the companion $p\times p$ matrices. We note that $\setOfTransformskDims_{\text{exact}}$ depends on $p$, but since this is evident, we do not index $\setOfTransformskDims_{\text{exact}}$ by $p$ to avoid cumbersome notations. From the standpoint of statistical applications, our focus in this paper is mostly on the case $k=1$ (which corresponds to ``standard" kernel methods commonly used in statistical learning) and $k=2$.

We show in Proposition \ref{prop:cardinalSetOfDiscretizedRotations} that if $\transform \in \setOfTransforms_{\text{exact}}$, $\transform$ is an orthogonal $p\times p$ matrix. Furthermore, we show in Proposition \ref{prop:cardinalSetOfDiscretizedRotations} that $\card{\setOfTransforms_{\text{exact}}}$ is polynomial in $p$. 
We therefore have the following proposition. 
\begin{proposition}\label{prop:caseOfExactRotations}
Let 
\begin{align*}
d_{ij,\text{noisy}}^2:=\inf_{\transformkDims\in\setOfTransformskDims_\text{exact}} \norm{\noisyImage_i-\transformkDims\circ \noisyImage_j}^2\;,\quad
d_{ij,\text{clean}}^2:=\inf_{\transformkDims\in\setOfTransformskDims_\text{exact}} \norm{\cleanImage_i-\transformkDims \circ \cleanImage_j}^2\;.
\end{align*}
Suppose $N_i$ are independent with $N_i^\vee\sim {\cal N}(0,\Sigma_i)$.  When Assumption G1 holds with $\setOfTransforms_\text{exact}$ being the set of companion matrices of $\setOfTransformskDims_\text{exact}$, we have
	$$
	\sup_{i\neq j}\left|d_{ij,\text{noisy}}^2-d_{ij,\text{clean}}^2-\trace{\Sigma_i+\Sigma_j}\right|=\lo_P(1)\;,
	$$	
	and
$$
	\sup_{\transformkDims \in \setOfTransformskDims_\text{exact} }\sup_{i\neq j} |\dijNoisy^2(\transformkDims)-\dijClean^2(\transformkDims)-\trace{\Sigma_i+\Sigma_j}|=\bO_P(u_{n,p})\;=\lo_P(1)\;.
$$
	
\end{proposition}

\subsubsection{On the transform $\transformkDims^*_{ij,\text{noisy}}$}
We now use the notations 
$$
\dijNoisy(\transformkDims)=\norm{I_i-\transformkDims\circ I_j}\;, \text{ and }
\dijClean(\transformkDims)=\norm{S_i-\transformkDims\circ S_j}\;.
$$
Naturally, the study of 
\begin{equation}\label{eq:defOptimalTransform}
\transformkDims^*_{ij,\text{noisy}}=\argmin_{\transformkDims \in \setOfTransformskDims_\text{exact}} \dijNoisy(\transformkDims)
\end{equation}
is more complicated than the study of $\inf_{\transformkDims \in \setOfTransformskDims_\text{exact}} \dijNoisy(\transformkDims)$. 
We will assume that the clean images are nicely behaved when it comes to the $\dijClean(\transformkDims)$ minimization, in that rotations that are near minimizers of $\dijClean(\transformkDims)$ are close to one another. More formally, we assume the following. 
\paragraph{Assumption A0 : } $\setOfTransformskDims_\text{exact}$ is a subset of $SO(k)$ and contains only exact rotations. Call $\transformkDims^*_{ij,\text{clean}}:=\argmin_{\transformkDims \in \setOfTransformskDims_\text{exact}} \dijClean^2(\transformkDims)$ and call $\setOfTransformskDims_{ij,\eps}:=\left\{\transformkDims \in \setOfTransformskDims_\text{exact}: \, \dijClean^2(\transformkDims)\leq \dijClean^2(\transformkDims^*_{ij,\text{clean}})+\eps\right\}$. We assume that
$$
\exists\delta_{ij,p}>0: \, \forall \eps<\delta_{ij,p}\, \forall \transformkDims \in \setOfTransformskDims_{ij,\eps}\,,\; d(\transformkDims,\transformkDims^*_{ij,\text{clean}})\leq g_{ij,p}(\eps)\;,
$$
for $d$ the canonical metric on the orthogonal group and some positive $g_{ij,p}(\eps)$.

\paragraph{Assumption A1 : } $\delta_{ij,p}$ can be chosen independently of $i,j$ and $p$. Furthermore, there exists a function $g$ such that $g(\eps)\tendsto 0$ as $\eps\tendsto 0$ and $g_{ij,p}(x)\leq g(x)$, if $x\leq \delta_{ij,p}\leq \delta$. 

We discuss the meaning of these assumptions after the statement and proof of the following theorem.
\begin{theorem}\label{thm:consistencyRotations}
Suppose that the assumptions underlying Theorem \ref{thm:controlApproxUnifIndicesAndTransfo} hold and that Assumptions G1, A0 and A1 hold. Suppose further that $\setOfTransformskDims_\text{exact}$ is the set of exact rotations for our discretization. Then, for any $\eta$ given, where $0<\eta<1$, as $p$ and $n$ go to infinity, 
\begin{equation}\label{eq:controlDistanceNoisyOptRotAndCleanOptRot}
\sup_{i\neq j}d(\transformkDims^*_{ij,\text{noisy}},\transformkDims^*_{ij,\text{clean}})=\bO_P(g(u^{1-\eta}_{n,p}))\;,
\end{equation}
where $u_{n,p}$ is defined in (\ref{definition:unp}).
(Under Assumption G1, $u_{n,p}\tendsto 0$ as $n$ and $p$ tend to infinity.)
\end{theorem}
The informal meaning of this theorem is that under regularity assumptions on the set of clean images, the optimal rotation computed from the set of noisy images is close to the optimal rotation computed from the set of clean images. In other words, this step of the CGL procedure is robust to noise. 
\begin{proof}
Clearly, $\transformkDims^*_{ij,clean}$ is a minimizer of $L_{ij}(\transformkDims):=\dijClean^2(\transformkDims)+\trace{\Sigma_i+\Sigma_j}$, since the second term does not depend on $\transformkDims$. Naturally, if Assumptions A0 and A1 apply to $\dijClean^2(\transformkDims)$, they apply to $\dijClean^2(\transformkDims)+C$, for $C$ any constant. In particular, taking $C=\trace{\Sigma_i+\Sigma_j}$, we see that Assumptions A0 and A1 apply to the function $L_{ij}(\transformkDims)$. 

The approximation results of  Proposition \ref{prop:caseOfExactRotations} guarantee that, under Assumptions G1, A0 and A1, $\transformkDims^*_{ij,\text{noisy}}$ is a near minimizer of $\dijClean^2(\transformkDims)$. Indeed, we have by definition, 
\begin{equation}\label{eq:almostTautologyTransformNoisy}
\dijNoisy^2(\transformkDims^*_{ij,\text{noisy}})\leq \dijNoisy^2(\transformkDims^*_{ij,\text{clean}}). 
\end{equation}
But under assumption G1, Proposition \ref{prop:caseOfExactRotations} and the fact that the elements of $\setOfTransforms_\text{exact}$ are orthogonal matrices imply that 
\begin{align}\label{eq:proof:22:dijLij}
\forall \transformkDims \in \setOfTransforms, \forall i \neq j \; \; \dijNoisy^2(\transformkDims)=L_{ij}(\transformkDims)+\bO_P(u_{n,p})\;.
\end{align}
Hence, we can rephrase Equation \eqref{eq:almostTautologyTransformNoisy} as 
\begin{align}\label{eq:proof:22:LijLij}
L_{ij}(\transformkDims^*_{ij,\text{noisy}})\leq L_{ij}(\transformkDims^*_{ij,\text{clean}})+\bO_P(u_{n,p})\;.
\end{align}
Indeed, by plugging (\ref{eq:proof:22:dijLij}) into (\ref{eq:almostTautologyTransformNoisy}), we have
$$
\forall i\neq j, \; \dijClean^2(\transformkDims^*_{ij,\text{noisy}})+\trace{\Sigma_i+\Sigma_j}\leq \dijClean^2(\transformkDims^*_{ij,\text{clean}})+\trace{\Sigma_i+\Sigma_j}+\bO_P(u_{n,p})\;.
$$ 
Now, by definition of $\dijClean^2$, we have 
\begin{align}\label{eq:proof:22:dijcleandijnoisy}
\dijClean^2(\transformkDims^*_{ij,\text{clean}})\leq \dijClean^2(\transformkDims^*_{ij,\text{noisy}})\;.
\end{align}
So by (\ref{eq:proof:22:LijLij}) and (\ref{eq:proof:22:dijcleandijnoisy}), we have shown that 
\begin{align*}
\forall i\neq j, \; \dijClean^2(\transformkDims^*_{ij,\text{clean}})+\trace{\Sigma_i+\Sigma_j}&\leq \dijClean^2(\transformkDims^*_{ij,\text{noisy}})+\trace{\Sigma_i+\Sigma_j} \;,\\
&\leq \dijClean^2(\transformkDims^*_{ij,\text{clean}})+\trace{\Sigma_i+\Sigma_j}+\bO_P(u_{n,p})\;.
\end{align*}
This clearly implies that 
\begin{align*}
\forall i\neq j, \; \dijClean^2(\transformkDims^*_{ij,\text{clean}})&\leq \dijClean^2(\transformkDims^*_{ij,\text{noisy}})\leq \dijClean^2(\transformkDims^*_{ij,\text{clean}})+\bO_P(u_{n,p})\;.
\end{align*}
Since $u_{n,p}\tendsto 0$ as $n$ and $p$ grow, this means that, for any given $\eta$, with $0<\eta<1$, with very high probability, 
$$
\forall 1\leq i \neq j \leq n \;, \transformkDims^*_{ij,\text{noisy}}\in \setOfTransformskDims_{ij,u_{n,p}^{1-\eta}}\;.
$$
We conclude, using Assumption A0, that with very high-probability, 
$$
\forall 1\leq i \neq j \leq n \;, d(\transformkDims^*_{ij,\text{noisy}},\transformkDims^*_{ij,\text{clean}})\leq g(u_{n,p}^{1-\eta})\;.
$$
\end{proof}

\paragraph{Interpretation of Assumptions A0-A1} Assumption A0 guarantees that all near minimizers of $\dijClean(\transformkDims)$ are close to one another and hence the optimum. Our uniform bounds in Proposition \ref{prop:caseOfExactRotations} only guarantee that $\transformkDims^*_{ij,\text{noisy}}$ is a near minimizer of $\dijClean(\transformkDims)$ and nothing more. If $\dijClean(\transformkDims)$ had near minimizers that were far from the optimum $\transformkDims^*_{ij,\text{clean}}$, it could very well happen that $\transformkDims^*_{ij,\text{noisy}}$ end up being close to one of these near minimizers but far from $\transformkDims^*_{ij,\text{clean}}$, and we would not have the consistency result of Theorem \ref{thm:consistencyRotations}. Hence, the robustness to noise of this part of the CGL algorithm is clearly tied to some regularity or ``niceness'' property for the set of clean images. 

In the cryo-EM problem, these assumptions reflect a fundamental property of a manifold dataset -- {\it its condition number} \cite{Niyogi_Smale_Weinberger:2009}. Conceptually, the condition number reflects ``how difficult it is to reconstruct the manifold'' from a finite sample of points from the manifold. Precisely, it is the inverse of the reach of the manifold, which is defined to be the radius of the smallest normal bundle that is homotopic to the manifold. This also highlights the fact that even if we were to run the CGL algorithm on the clean dataset, without these assumptions, the results might not be stable and reliable since intrinsically distant points (i.e distant in the geodesic distance) might be identified as neighbors.

\paragraph{About $\setOfTransformskDims_\text{exact}$ and extensions} We are chiefly interested in this paper about 2-dimensional images and hence about the case $k=2$ (see the cryoEM example). It is then clear that when our polar coordinate grid is fine, $\setOfTransformskDims_\text{exact}$ is also a fine discretization of $SO(2)$ and contains many elements. (More details are given in Subsection \ref{subsec:cardExactRotations}.) The situation is more intricate when $k\geq 3$, but since it is a bit tangential to the main purpose of the current paper, we do not discuss it further here. We refer the interested reader to Subsection \ref{subsec:cardExactRotations} for more details about the case $k\geq 3$. 

We also note that our arguments are not tied to using a standard polar coordinate grid for the discretization of the images. {For another sampling grid, we would possibly get another
$\setOfTransformskDims_\text{exact}$}. Our arguments go through when : a) if $\transformkDims \in \setOfTransformskDims$, the operation $\transformkDims\circ $ maps our sampling grid of points onto itself; b) $\card{\setOfTransformskDims}$ grows polynomially in $p$.

\subsubsection{Extensions and different approaches}

At the gist of our arguments are strong concentration results for quadratic forms in Gaussian random variables. Naturally, our results extend to other types of random variables for which these concentration properties hold. We refer to \cite{ledoux2001} and  \cite{nekCorrEllipD} for examples. A natural example in our context would be a situation where $N_i=\Sigma_i^{1/2} X_i$, and $X_i$ has i.i.d uniformly bounded entries. This is particularly relevant in the case where $\Sigma_i$ is diagonal for instance  - the interpretation being then that the noise contamination is through the corruption of each individual pixel by independent random variables with possibly different standard deviations. The arguments in Lemma \ref{lemma:controlSupQuadForms} handle this case, though the bound is slightly worse than the one in Lemma \ref{lemma:controlSupQuadFormsLaurentMassart} when a few eigenvalues of $\Sigma_i$ are larger than most of the others. Indeed, the only thing that matters in this more general analysis is the largest eigenvalue of $\Sigma_i$, so that in the notation of Assumption \textbf{G1}, $\sqrt{p}s_p^2$ is replaced by $\sqrt{p}\sigma_p^2$. Hence, our approximation will require in this more general setting that $\sigma_p=\lo(p^{-1/4})$, whereas we have seen in the Gaussian case that we can tolerate a much larger largest eigenvalue. 

We also note that we could of course settle for weaker results on concentration of quadratic forms, which would apply to more distributions. For instance, using bounds on $\Exp{|\norm{N_i}^2-\Exp{\norm{N_i}^2}|^{k}}$ would change the dependence of results such as Proposition \ref{prop:controlApproxdijNoisy} on $\card{\setOfTransforms}n^2$ from powers of logarithm to powers of $1/k$. This is in turn would mean that our results would become tolerant to lower levels of noise but apply to more noise distributions.

\subsection{Consequences for CGL algorithm and other kernel-based methods}\label{subsec:consequencesForVDM}
\subsubsection{Reminders and preliminaries}
Recall that in CGL methods performed with the rotationally invariance distance - henceforth RID - induced by $SO(k)$, we mostly care about the spectral properties - especially large eigenvalues and corresponding eigenvectors - of the CGL matrix 
$L(\widetilde{W},\widetilde{G})$,
where $\widetilde{W}$ is a $n\times n$ matrix and $\widetilde{G}$ is a $nk\times nk$ block-matrix  with $k\times k$ blocks defined through 
$$
\widetilde{W}_{i,j}=\exp(-\dijNoisy^2/\eps),\quad \widetilde{G}_{i,j}=\mathtt{O}^*_{ij,\text{noisy}}\;,
$$
where $\mathtt{O}^*_{ij,\text{noisy}}$ is defined in Equation \eqref{eq:defOptimalTransform}.

The ``good'' properties of CGL stem from the fact that the matrix $L(W,G)$, the CGL matrix associated with the clean images, has ``good" spectral properties. For example, when a manifold structure is assumed, the theoretical work of \cite{singer_wu:2012,singer_wu:2013} relates the properties of $L(W,G)$ - the matrix obtained in the same manner as above when we replace $\dijNoisy$ by $\dijClean$ and $\mathtt{O}^*_{ij,\text{noisy}}$ by $\mathtt{O}^*_{ij,\text{clean}}$ - to the geometric and topological properties of the manifold from which the data is sampled. The natural approximate ``sparsity''  of the spectrum of this kind of matrices is discussed in Section \ref{Appendix:Section:CGL}.

In practice, the data analyst has to work with $L(\widetilde{W},\widetilde{G})$. Hence, it could potentially be the case that $L(\widetilde{W},\widetilde{G})$ does not share many of the good properties of $L(W,G)$. Indeed, we explain below that this is in general the case and propose a  modification to the standard algorithm to make the results of CGL methods more robust to noise. 
All these arguments suggest that it is natural to study the properties of the standard CGL algorithm applied to noisy data. 

We mention that CGL algorithms may apply beyond the case of the rotational invariance distance and $O(k)$ and we explain in Subsubsection \ref{subsubsec:CGLbeyondSOk} how our results apply in this more general context.

\subsubsection{Modified CGL algorithm and rotationally invariant distance}
We now show that our modification to the standard algorithm is robust to noise. More precisely, we show that the modified CGL matrix $L_0(\widetilde{W},\widetilde{G})$ is spectrally close to the CGL matrix computed from the noise-free data, $L(W,G)$.

\begin{proposition}\label{prop:connectionGraphLaplacianApproxModifS}
Consider the modified CGL matrix $L_0(\widetilde{W},\widetilde{G})$ computed from the noisy data and the CGL matrix $L(W,G)$ computed from the noise-free data. Under Assumptions \textbf{G1} and \textbf{A0}-\textbf{A1}, we have, if $\trace{\Sigma_i}=\trace{\Sigma_j}=\trace{\Sigma}$ for all $(i,j)$, 
$$
\opnorm{L_0(\widetilde{W},\widetilde{G})-L(W,G)}=\lo_P(1)\;,
$$
provided there exists $\gamma>0$, independent of $n$ and $p$ such that 
$$
\inf_i \sum_{j\neq i} \frac{\exp(-\dijClean^2/\eps)}{n}\geq \gamma>0\;.
$$
\end{proposition}
Note that the previous result means that $L_0(\widetilde{W},\widetilde{G})$ and $L(W,G)$ are essentially spectrally equivalent: indeed we can use the Davis-Kahan theorem or Weyl's inequality to relate eigenvectors and eigenvalues of $L_0(\widetilde{W},\widetilde{G})$ to those of $L(W,G)$ (see \cite{stewart90}, \cite{bhatia97} or \cite{nekSparseMatrices} for a brief discussion putting all the needed results together). In particular, if the large eigenvalues of $L(W,G)$ are separated from the rest of the spectrum, the eigenvalues of $L_0(\widetilde{W},\widetilde{G})$ and corresponding eigenspaces will be close to those of $L(W,G)$.

\begin{proof}
The proposition is a simple consequence of our previous results and Lemma \ref{lemma:approxCGLMatGeneralModif} above. 
Indeed, in the notation of Lemma \ref{lemma:approxCGLMatGeneralModif}, we call 
$$
w_{i,j}=\begin{cases}
\exp(-\dijClean^2/\eps)& \text{ if } i\neq j \\
1 &\text{ if } i=j 
\end{cases}\; 
\quad\text{and}\quad
\tilde{w}_{i,j}=\begin{cases}
\exp(-\dijNoisy^2/\eps) &\text{ if } i\neq j \\
1 &\text{ if } i=j 
\end{cases}\;.
$$
Similarly, we call 
$$
G_{i,j}=\begin{cases}
\mathtt{O}^*_{ij,clean} &\text{ if } i\neq j\\
\id_d &\text{ if } i=j
\end{cases}
\quad\text{and}\quad
\widetilde{G}_{i,j}=\begin{cases}
\mathtt{O}^*_{ij,noisy} &\text{ if } i\neq j\\
\id_d &\text{ if } i=j
\end{cases}\;.
$$
Under Assumption G1, we know that, if $f_i=\exp(-2\trace{\Sigma}/\eps)$, $\sup_{i\neq j}|w_{i,j}-\tilde{w}_{i,j}/f_i|=\lo_P(1)$. Similarly, under Assumptions G1, A0 and A1, we know, using Theorem \ref{thm:consistencyRotations} that 
$$
\sup_{i,j}d(G_{i,j},\widetilde{G}_{i,j})=\lo_P(1)
$$
and therefore, since $k$, the parameter of $SO(k)$, is held fixed in our asymptotics, 
$$
\sup_{i,j}\norm{G_{i,j}-\widetilde{G}_{i,j}}_F=\lo_P(1)\;.
$$
Since we assumed that $$
\inf_i \sum_{j\neq i} \frac{\exp(-\dijClean^2/\eps)}{n}\geq \gamma>0\;,
$$
i.e, in the notations of Lemma \ref{lemma:approxCGLMatGeneralModif}
$$
\inf_i \frac{\sum_{j\neq i} w_{i,j}}{n}\geq \gamma>0\;,
$$
where $\gamma$ is independent of $n$ and $p$,
all the assumptions of Lemma \ref{lemma:approxCGLMatGeneralModif} are satisfied when $n$ and $p$ are large enough, and we conclude that, in the notations of this lemma, 
$$
\opnorm{L_0(W,G)-L_0(\widetilde{W},\widetilde{G})}=\lo_P(1)\;.
$$

Furthermore, we have $0\leq w_{i,j},\tilde{w}_{i,j}\leq 1$, $\norm{G_{i,j}}_F\leq \sqrt{k}$ and $\norm{\widetilde{G}_{i,j}}_F\leq \sqrt{k}$, the latter two results coming from the fact that the columns of $G_{i,j}$ and $\widetilde{G}_{i,j}$ have unit norm. 
So we conclude that 
$$
\opnorm{L(W,G)-L_0(\widetilde{W},\widetilde{G})}=\lo_P(1)\;.
$$
\end{proof}

\paragraph{Is the modification of the algorithm really needed?} It is natural to ask what would have happened if we had not modified the standard algorithm, i.e if we had worked with $L(\widetilde{W},\widetilde{G})$ instead of $L_0(\widetilde{W},\widetilde{G})$. It is easy to see that 
$$
L(\widetilde{W},\widetilde{G})=L_0(\widetilde{W},\widetilde{G})+\mathsf{D}
$$
where $\mathsf{D}$ is a block diagonal matrix with 
$$
\mathsf{D}(i,i)=\frac{\tilde{w}_{i,i}}{\sum_{j\neq i}\tilde{w}_{i,j}}\id_k=\frac{1}{\sum_{j\neq i}\tilde{w}_{i,j}}\id_k\;.
$$
Under our assumptions, 
$$
\opnorm{n\exp(-2\trace{\Sigma}/\eps) \mathsf{D}-D\left(\left\{\frac{\sum_{j\neq i} \exp(-\dijClean^2/\eps)}{n}\right\}_{i=1}^n\right)}=\lo_P(1)\;.
$$
We also recall that under Assumption \textbf{G1}, $\trace{\Sigma}$ can be as large as $p^{1/2-\eta}$ - a very large number in our asymptotics. So in particular, if $n$ is polynomial in $p$, we have then $n^{-1}\exp(2\trace{\Sigma}/\eps)\tendsto \infty$. This implies that 
$$
L(\widetilde{W},\widetilde{G})=L_0(\widetilde{W},\widetilde{G})+\mathsf{D}
$$
is then dominated in spectral terms by $\mathsf{D}$. So it is clear that in the high-noise regime, if we had used the standard CGL algorithm, the spectrum of $L(\widetilde{W},\widetilde{G})$ would have mirrored that of $\mathsf{D}$ - which has little to do with the spectrum of $L(W,G)$, which we are trying to estimate - and the noise would have rendered the algorithm ineffective. 

By contrast, by using the modification we propose, we guarantee that even in the high-noise regime, the spectral properties of $L_0(\widetilde{W},\widetilde{G})$ mirror those of $L(W,G)$. We have hence made the CGL algorithm more robust to noise.

\textbf{On the use of nearest neighbor graphs} In practice, variants of the CGL algorithms we have described use nearest neighbor information to replace $w_{i,j}$ by 0 if $w_{i,j}$ is not among the $k$ largest elements of $\{w_{i,j}\}_{j=1}^n$. In the high-noise setting, the nearest-neighbor information is typically not robust to noise, which is why we proposed to use all the $w_{i,j}$'s and avoid the nearest neighbor variant of the CGL algorithm, even though the latter probably makes more intuitive sense in the noise-free context. A systematic study of the difference between these two variants is postponed to future work.

\textbf{Comparison with previous results in the literature} As far as we know, the study of the impact of high-dimensional additive noise on kernel methods was started in \cite{nekInfoPlusNoiseKernelMatrices10}. Compared to this paper, our extension is two-fold: 1) the noise level (i.e $\trace{\Sigma}$) that is studied in the current paper is much higher than what was studied in \cite{nekInfoPlusNoiseKernelMatrices10}. This is partly a result of the fact that the current paper focuses on the Gaussian kernel whereas \cite{nekInfoPlusNoiseKernelMatrices10} studied many more kernels. 2) \cite{nekInfoPlusNoiseKernelMatrices10} focused on standard kernel methods - based on the graph Laplacian, such as  diffusion maps -  where the connection information is not included in the data analysis. Incorporating this new element creates new difficulties. In other respects, we also refer to \cite{singer:2011} for another study of the influence of noise in a different setup.

\subsubsection{CGL beyond the rotational invariance distance}\label{subsubsec:CGLbeyondSOk}
The previous analysis has been carried out for the RID and corresponding rotations for whom we studied the impact of additive noise in Subsection \ref{subsec:impactNoiseVDMProcedure}. However, it is clear that our results apply much more broadly. We have the following proposition. 

\begin{proposition}\label{prop:generalCGLApproxScheme}
Suppose we are given a collection $\mathsf{d}_{i,j,noisy}$ of (scalar-valued) dissimilarities between noisy versions of objects $i$ and $j$, $1\leq i,j\leq n$. Suppose objects $i$ and $j$ have (scalar-valued) dissimiliarity $\mathsf{d}_{i,j,clean}$. Consider the asymptotic regime where $n\tendsto \infty$ and suppose that there exists $\xi_n \in \mathbb{R}$ such that 
$$
\sup_{i\neq j}|\mathsf{d}_{i,j,noisy}^2-\mathsf{d}^2_{i,j,clean}-\xi_n|=\lo_P(1)\;.
$$
Call $\tilde{w}_{i,j}=\exp(-\mathsf{d}_{i,j,noisy}^2/\nu)$ and $w_{i,j}=\exp(-\mathsf{d}^2_{i,j,clean}/\nu)$ the corresponding affinities. $\nu$ is held fixed in our asymptotics, though the way affinities are computed may change with $n$. 

Suppose $\widetilde{G}_{i,j}$ is the connection between noisy versions of objects $i$ and $j$ and $G_{i,j}$ is the connection between the clean version of objects $i$ and $j$. 
Suppose that $w_{i,j}$, $G_{i,j}$ and $\widetilde{G}_{i,j}$ satisfy the assumptions of Lemma \ref{lemma:approxCGLMatGeneralModif}, with $\eps$ and $\eta$ possibly random but $\lo_P(1)$ and $\gamma$ bounded below as $n\tendsto \infty$. Then
$$
\opnorm{L(W,G)-L_0(\widetilde{W},\widetilde{G})}=\lo_P(1)\;.
$$
\end{proposition}

\begin{proof}
This proposition is just a consequence of Lemma \ref{lemma:approxCGLMatGeneralModif}. Indeed, the affinities are all bounded by 1. Furthermore, we can use $f_i=\exp(-\xi_n/\nu)$ and all the approximation results needed in Lemma \ref{lemma:approxCGLMatGeneralModif} are true, so the result follows.  
\end{proof}

\subsubsection{A situation without robustness to noise}

So far, our work has been quite general and has shown that when the noise is Gaussian (or Gaussian-like) and its covariance $\Sigma_i$ is such that $\trace{\Sigma_i}=\trace{\Sigma_j}$ for all $i,j$, CGL algorithms can be made robust to noise. 

It has been recognized \cite{DiaconisFreedmanProjPursuit84,nekCorrEllipD,nekInfoPlusNoiseKernelMatrices10,NEKRobustPaperPNAS2013Published} that to study the robustness of various statistical procedures in high-dimension, it is essential to move beyond the Gaussian-like situation and study for instance elliptical/scale mixture of Gaussian models. This largely due to the peculiar geometry of high-dimensional Gaussian and Gaussian-like vectors (see above references and \cite{HallMarronNeemanJRSSb05}).

If we now write down a model for the noise where $N_i=\lambda_i Z_i$, where $Z_i$ are i.i.d ${\cal N}(0,\Sigma)$, $\lambda_i$'s are i.i.d with $\Exp{\lambda_i^2}=1$ and $\lambda_i \in \mathbb{R}$ is independent of $Z_i$, it is easy to modify our analysis (assuming for instance that $\lambda_i^2$ are bounded, though this condition could easily be relaxed) and to realize that our main approximation result in Proposition \ref{prop:controlApproxdijNoisy} is replaced by 
$$
\sup_{i\neq j}\left|d_{ij,\text{noisy}}^2-d_{ij,\text{clean}}^2-[\lambda_i^2+\lambda_j^2]\trace{\Sigma}\right|=\lo_P(1)\;.
$$
In this situation, Theorem \ref{thm:consistencyRotations} is still valid. However, Proposition \ref{prop:connectionGraphLaplacianApproxModifS} is not valid anymore. The matrix $L_0(\widetilde{W},\widetilde{G})$ can be approximated by a matrix that depends both on the signal and the distribution of the $\lambda_i^2$'s. And there is no guarantee in general that this matrix will have approximately the same spectral properties as $L(W,G)$ or $L_0(W,G)$, the  CGL matrix generated from the noise-free signals. 
This suggests that even our modification of the original CGL algorithm will not be robust to this ``elliptical''-noise contamination.

\section{Numerical work}
Although the robustness properties of CGL methods were not well studied in the past, these methods have been successfully applied to different problems; for example, \cite{singer_zhao_shkolnisky_hadani:2011,singer_wu:2012,Zhao_Singer:2013,Marchesini_Tu_Wu:2014,Cucuringu_Singer_Cowburn:2012,Alexeev_Bandeira_Fickus_Mixon:2013}. In this section, we show simulated examples to illustrate the practical performance of our theoretical findings about CGL methods. We refer interested readers to the aforementioned papers for details and results of its applications.

To demonstrate the main finding of this paper - that CGL methods are robust to high-levels of noise in the spectral sense - we take the noise to be a random Gaussian vector $Z\sim \mathcal{N}(0,cI_p/p^{\alpha})$, where $\alpha\leq 1$ and $c>0$. Note that the amount of noise, or the trace of the covariance matrix of $Z$, is $cp^{1-\alpha}$ and will blow up when $p\to \infty$ and $\alpha<1$.

\subsection{$1$-dim manifold}\label{Section:numerical:1dim}
Our first example is a dataset sampled from a low dimensional manifold, which is embedded in a high dimensional space. This dataset can be viewed as a collection of high dimensional points which is (locally) parametrized by only few parameters\footnote{By definition, although locally the manifold resembles Euclidean space near a point, globally it might not. Thus, in general we can only parametrize the manifold locally. This feature captures the possible nonlinear structure in the data.}, but in a nonlinear way. 

As a concrete example, we take the twisted bell-shaped simple and closed curve, denoted as $\manifold$, embedded in the first $3$ axes of $\RR^p$, where $p\gg 2$, via $\iota:[0,2\pi)\to \RR^p$:
\[
\iota:\,t\mapsto [\,\cos(t),\,(1-0.8e^{-8\cos^2t})\cos(\pi(\cos(t)+1)/4),\,(1-0.8e^{-8\cos^2t})\sin(\pi(\cos(t)+1)/4),\,0,\ldots,0\,]\in\RR^p\,.
\] 
$\manifold$ is a 1-dim smooth manifold without boundary; that is, no matter how big $p$ is, locally the points on $\manifold$ can be parametrized by only $1$ parameter. See Figure \ref{fig:numerical1:clean} (A) for an illustration. We mention that one interesting dataset of this kind is the 2-D tomography from noisy projections taken at unknown random directions \cite{singer_wu:2013a}.

For our numerical work, we independently sample $n$ points uniformly at random from $[0,2\pi)$. Due to the non-linear nature of $\iota$, it is equivalent to non-uniformly sampling $n$ points from $\manifold$ independently. Denote the clean data as $\mathcal{Y}=\{y_i\}_{i=1}^n\subset \manifold$. The data $\mathcal{X}=\{x_i\}_{i=1}^n$ we analyze is the clean data contaminated by noise, i.e $x_i=y_i+Z_i$, with $Z_i$ i.i.d with the same distribution as $Z$. We measure the {\it signal-to-noise ratio} of the dataset by the quantity $\text{snrdb}:=20\log\frac{\sqrt{\EE X^TX}}{\sqrt{\EE Z^TZ}}$. We take $n=p=1000$ and $\alpha=1/4$. Note that $\alpha=1/4$ is the critical value in our analysis beyond which our results do not apply. It corresponds to a high-noise level; for example, the snrdb will be $-9.25$ and $-18.73$ respectively when $c=0.25,\,0.4$.

Then, we build up the connection graph by setting $\graphV:=\mathcal{X}$ and $\graphE:=\{(i,j);\,i,j\in \graphV\}$. Note that in practice, it is common to use a nearest-neighbor scheme to build up the graph, denoted as $\graphG^{\text{NN}}$, for the sake of computational efficiency. However, since the sets of nearest neighbors are not stable under the action of the noise, we also consider here the complete graph scheme, denoted as $\graphG$.  Next we assign the weight function as $w:(i,j)\mapsto e^{-\|x_i-x_j\|^2_{\RR^p}/m}$, where $m$ is the $25\%$ quantile of all Euclidean distances between pairs of $(x_i,x_j)\in\graphE$, and the connection function is defined to be a trivial one, that is, $\relationship(i,j)=1$ for all $(x_i,x_j)\in\graphE$. 

With the connection graph, we build up the CGL matrix (in this $1$-dim manifold with the trivial connection, it is equivalent to the graph Laplacian (GL)) from $\graphG^{\text{NN}}$ and $\graphG$ for comparison, denoted as $L^{\text{NN}}(\widetilde{W},\widetilde{G})$ and $L(\widetilde{W},\widetilde{G})$ respectively (see (\ref{definition:LWG})).
We have seen in the analysis described earlier in the paper that, when $\alpha<1$, we have to remove the diagonal terms of the CGL matrix in order to preserve spectral properties. So, we also consider the matrix $L_0(\widetilde{W},\widetilde{G})$ for the comparison. 

We then evaluate the eigenvalues and eigenvectors of the above three different CGL's. To simplify the notation, we use the same notations to denote the eigenvectors $u_1,u_2,u_3\ldots\in\RR^n$ associated with the eigenvalues $1=\lambda_1>\lambda_2\geq \lambda_3\geq\ldots\geq 0$. We now show two sets of results to demonstrate the robustness of the CGL methods studied in this paper.

\underline{Dimension Reduction and Data Visualization}: To achieve this, we may embed the sampled points into $\RR^m$ by the {\it truncated diffusion maps (tDM) with time $t>0$ and precision $\delta>0$}:
\[
\Phi_{t,m}:\,x_i\mapsto (\lambda_2^t u_2(i),\, \lambda_3^t u_3(i),\,\ldots, \lambda_{m+1}^t u_{m+1}(i))\in \RR^m,
\]
where $\lambda_{m+1}>\delta$ and $\lambda_{m+2}\leq \delta$; that is, we map the $i$-th data point to $\RR^m$ using the first $m$ non-trivial eigenvectors of the CGL. We choose $\delta=0.2$ in this simulation. The embedding results of $\mathcal{Y}$, $\Phi_{1,3}$, based on the above different CGL's are shown in Figure \ref{fig:numerical1:clean}, and the results from $\mathcal{X}$ with $c=0.4$ are shown in Figure \ref{fig:numerical1:smallNoise}. Ideally, we would expect to recover the ``parametrization'' of the dataset by the idea that the eigenvectors of the CGL represent a set of new coordinates for the data points, so the high dimensional dataset can be visualized in this new set of coordinates or its dimension can be reduced. In this specific example, we would expect to find a simple and closed curve out of the noisy dataset which represents the dataset in $\RR^3$. Clearly when the dataset is clean, we succeed in the task no matter which CGL we use. However, if the dataset is noisy, at high-noise levels, the embedding might not be that meaningful if we use $L^{\text{NN}}(\widetilde{W},\widetilde{G})$ or $L(\widetilde{W},\widetilde{G})$. Indeed, as shown in Figure \ref{fig:numerical1:smallNoise},  with $L^{\text{NN}}(\widetilde{W},\widetilde{G})$ the structure of the dataset is barely recovered; with $L(\widetilde{W},\widetilde{G})$, even though we can get the simple closed curve\footnote{The main idea behind tDM is embedding the dataset to a lower dimensional Euclidean space so that the structure underlying the data can be extracted. Please see Section \ref{subsection:Appendix:VDMDM} for details.} back, there are several outliers which might deteriorate the interpretation. In this noisy case, we can only succeed in the task if we choose $L_0(\widetilde{W},\widetilde{G})$, as is discussed in this paper.

\begin{figure}
\begin{center}
\includegraphics[width=0.9\columnwidth]{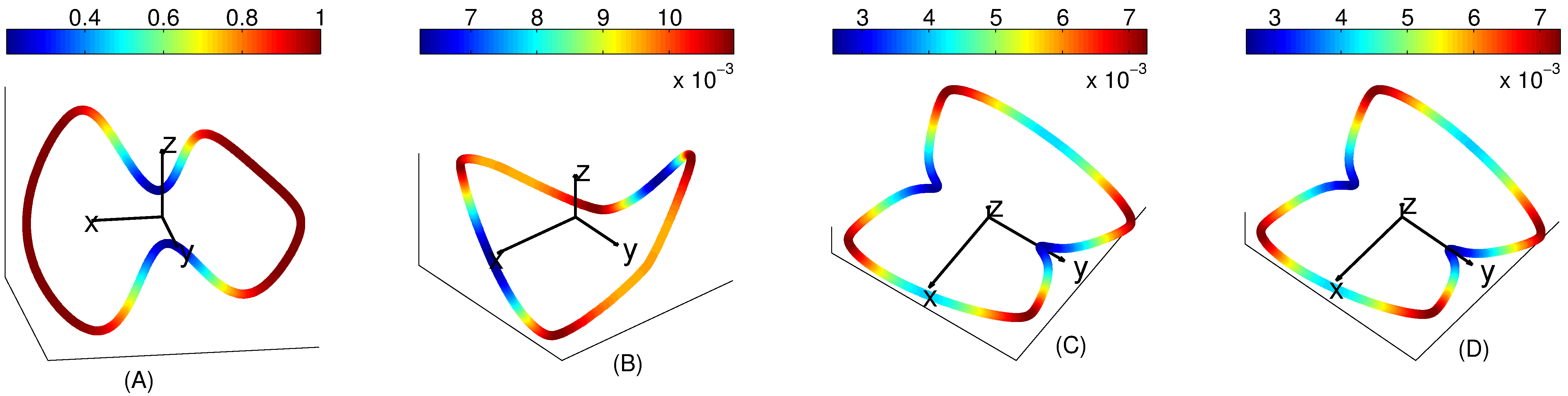}
\end{center}
\caption{\small 
Clean samples from the twisted bell-shaped manifold. (A): the clean samples. Here we only plot the first $3$ axes of the high dimensional data $\mathcal{Y}$. The color of each point is a surrogate of the norm of each embedded point -- blue means a relative small norm and dark red means a relative large norm; the scale above the figure refers to $\{\norm{x_i}_2\}_{i=1}^n$, i.e the norm of the data vectors in $\mathbb{R}^p$.(B): the results of the truncated diffusion maps (tDM), $\Phi_{1,3}$, when the connection graph is $\graphG^{\text{NN}}$ and the diagonal entries are not removed, where the number of nearest neighbors is chosen to be $100$; (C): the result of tDM, $\Phi_{1,3}$, when the connection graph is $\graphG$ and the diagonal entries are not removed; (D): the result of tDM, $\Phi_{1,3}$, when the connection graph is $\graphG$ and the diagonal entries are removed. Note that without surprise, the ``parametrization'' of the bell shaped manifold is recovered in (B), (C) and (D). For (B), (C), and (D), the scales above the figures refer to the norm of $\{\Phi_{1,3}(x_i)\}_{i=1}^n$; those vectors are of course 3-dimensional, which explains the difference in magnitude of our scales.
}\label{fig:numerical1:clean}
\end{figure}

\begin{figure}
\begin{center}
\includegraphics[width=0.9\columnwidth]{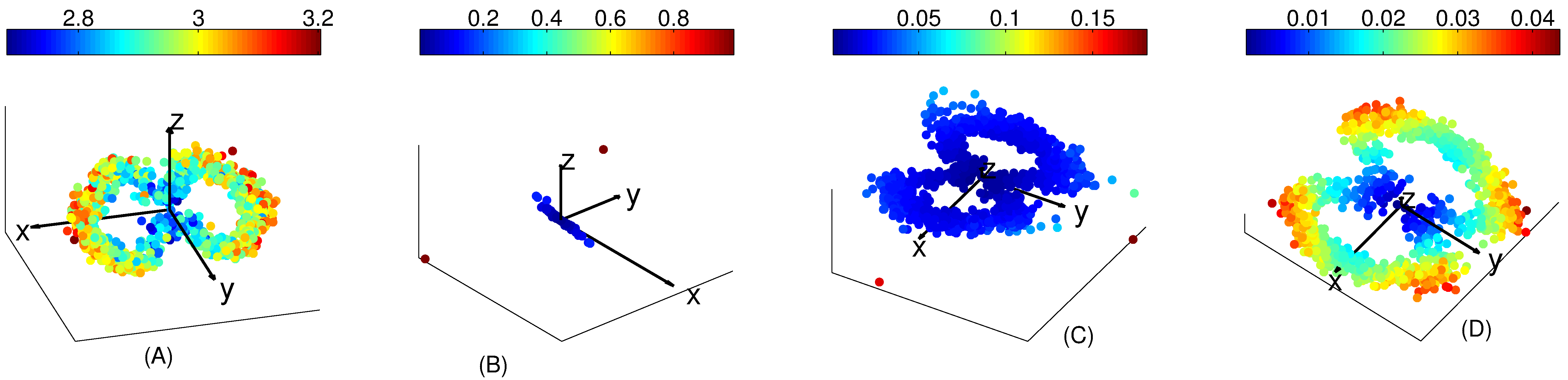}
\end{center}
\caption{\small Noisy samples from the twisted bell-shaped manifold with $\alpha=1/4$ and $c=0.4$. (A): the noisy samples. Note that we only plot the first $3$ axes of the data $\mathcal{X}$, so the ``big'' noise seems small, since 997 out of 1000 coordinates are not plotted. To emphasize the relationship among data points, the color of each point is a surrogate of the norm of each embedded point -- blue means a relative small norm and dark red means a relative large norm; the scale above the figure refers to $\{\norm{x_i}\}_{i=1}^n$, i.e the norm of our 1000-dimensional vectors. (B): the results of the truncated diffusion maps (tDM), $\Phi_{1,3}$, when the connection graph is $\graphG^{\text{NN}}$ and the diagonal entries are not removed, where the number of nearest neighbors is chosen to be $100$. We can barely see the circle structure in the middle, and there are several big outliers; (C): the result of tDM, $\Phi_{1,3}$, when the connection graph is $\graphG$ and the diagonal entries are not removed. Note that when compared with (B), the embedding is better in the sense that the ``parametrization'', the simple and close curve, is better recovered. But we can still observe several outliers; (D): the result of tDM, $\Phi_{1,3}$, when the connection graph is $\graphG$ and the diagonal entries are removed. Note that compared with (C), the embedding is yet better in the sense that the number of  outliers is reduced and the parametrization of the manifold is recovered. Note that for (B), (C),(D), the scale above the figures refer to $\{\norm{\Phi_{1,3}(x_i)}\}_{i=1}^n$, which are 3-dimensional vectors. The different scales indicate the presence of outliers. Compare also with the scales in Figure \ref{fig:numerical1:clean}, (B), (C), (D).}\label{fig:numerical1:smallNoise}
\end{figure}

\underline{Nearest Neighbors Estimation}
Estimating nearest neighbors of a given data point from a noisy dataset is not only important but also challenging in practice (for example, it is essential in the class averaging algorithm for the cryo-EM problem). This problem is directly related to  local geodesic distance estimation when the dataset is modeled by the manifold. Their theoretical properties make diffusion maps and vector diffusion maps particularly well-suited for these tasks. To determine the neighbors, we need the notion of distance. In addition to the naive $L^2$ distance between points, we consider the {\it diffusion distance} between two points $x_i,x_j\in\mathcal{X}$ by
\[
d_{\text{DD}}(x_i,x_j):=\|\Phi_{t,m}(x_i)-\Phi_{t,m}(x_j)\|_{\RR^m}.
\]
Then, we determine the nearest neighbors of each data point based on these distances, where we choose $t=1$ and $\delta=0.2$ for the diffusion distance. More precisely, we first determine $10$ nearest neighbors of $x_i$, denoted as $x_{i_j}$, $j=1,\ldots,10$, from the noisy dataset $\mathcal{X}$, for all $i$. Then, since we know the ground truth, we may check the true relationship between $y_{i}$ and $y_{i_{j}}$, $j=1,\ldots,10$, i.e $d_{\text{DD}}(y_i,y_{i_j})$ for various CGL methods, or $\norm{y_i-y_{i_j}}$ if we use $L^2$ distance. Clearly, if the method preserves nearest neighbor information, at least approximately, the ranks of the $y_{i_j}$'s measured in terms of distances to $y_i$ should be small. To quantify the estimation accuracy, we collect the ranks of all estimated nearest neighbors, and plot the cumulative distribution results in Figure \ref{figure:numerical1:nnEst}. In other words, if we call $R_{i_j}$ the rank of $y_{i_j}$ in terms of distance to $y_i$, we plot the cdf of $\{\{R_{i_j}\}_{j=1}^{10}\}_{i=1}^n$ for the various distances we use. (There are many other methods one could use to do these comparisons, such as using Kendall's $\tau$ and variants (see \cite{comparingTopKLists2003}). The one we use here has the benefit of simplicity.) When the dataset is clean, all methods perform the same, as is predicted in Theorem \ref{thm:geod:DM}. It is clear from the results that when the noise is large, the result based on the $L^2$ distance is much worse than the others. The performance based on the diffusion distance from $L^{\text{NN}}(\widetilde{W},\widetilde{G})$ is better when the noise level is not big, but still a non-negligible portion of error exists; the results based on $L(\widetilde{W},\widetilde{G})$ and $L_0(\widetilde{W},\widetilde{G})$ are much better, while the result based on $L_0(\widetilde{W},\widetilde{G})$ is the best.

In conclusion, in addition to showing the robustness of CGL to noise, we have demonstrated the spectrally close relationship between $L_0(\widetilde{W},\widetilde{G})$ and $L(W,G)$, which is proved in Proposition \ref{prop:connectionGraphLaplacianApproxModifS}.

\begin{figure}
\begin{center}
\includegraphics[width=0.24\columnwidth]{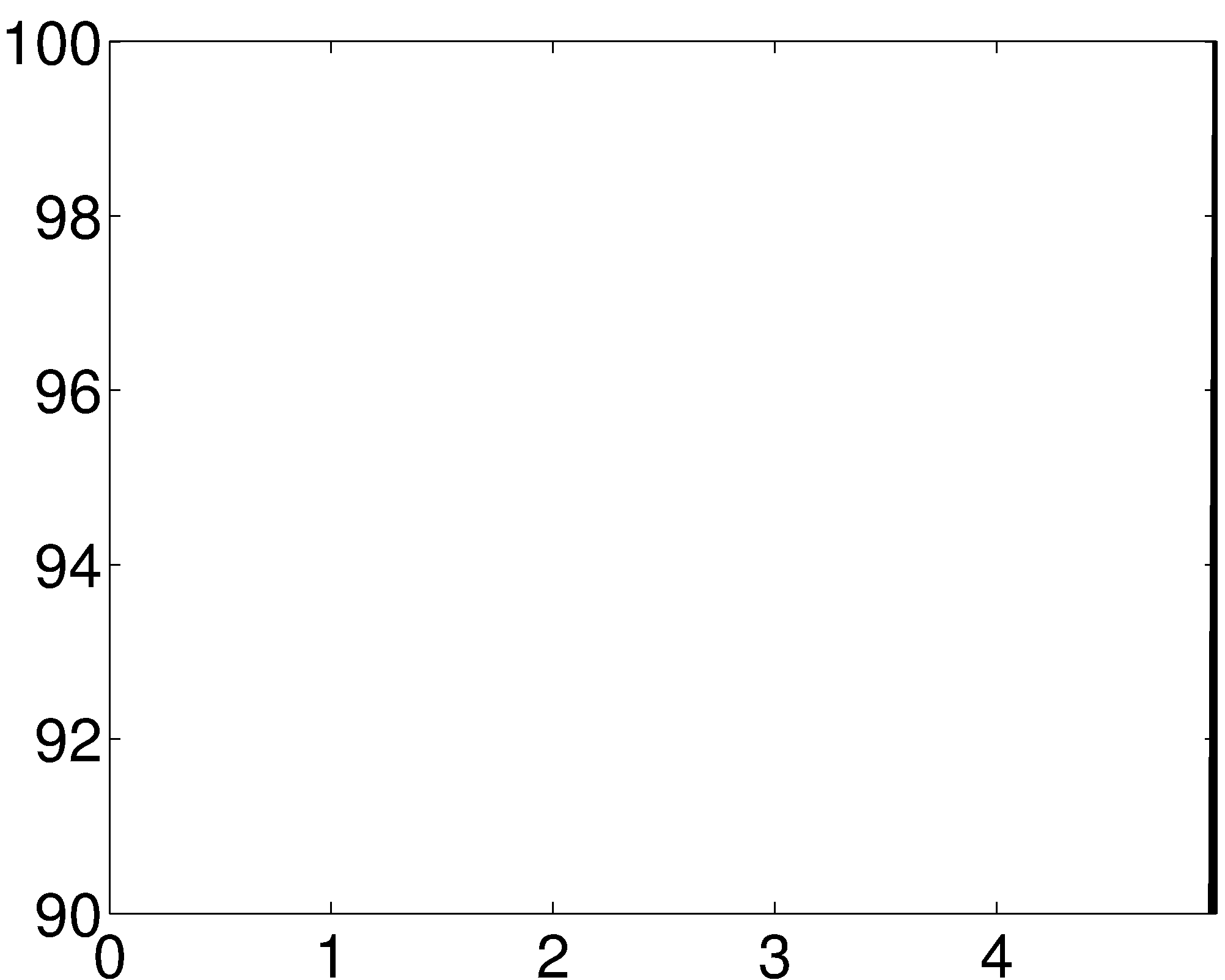}
\includegraphics[width=0.24\columnwidth]{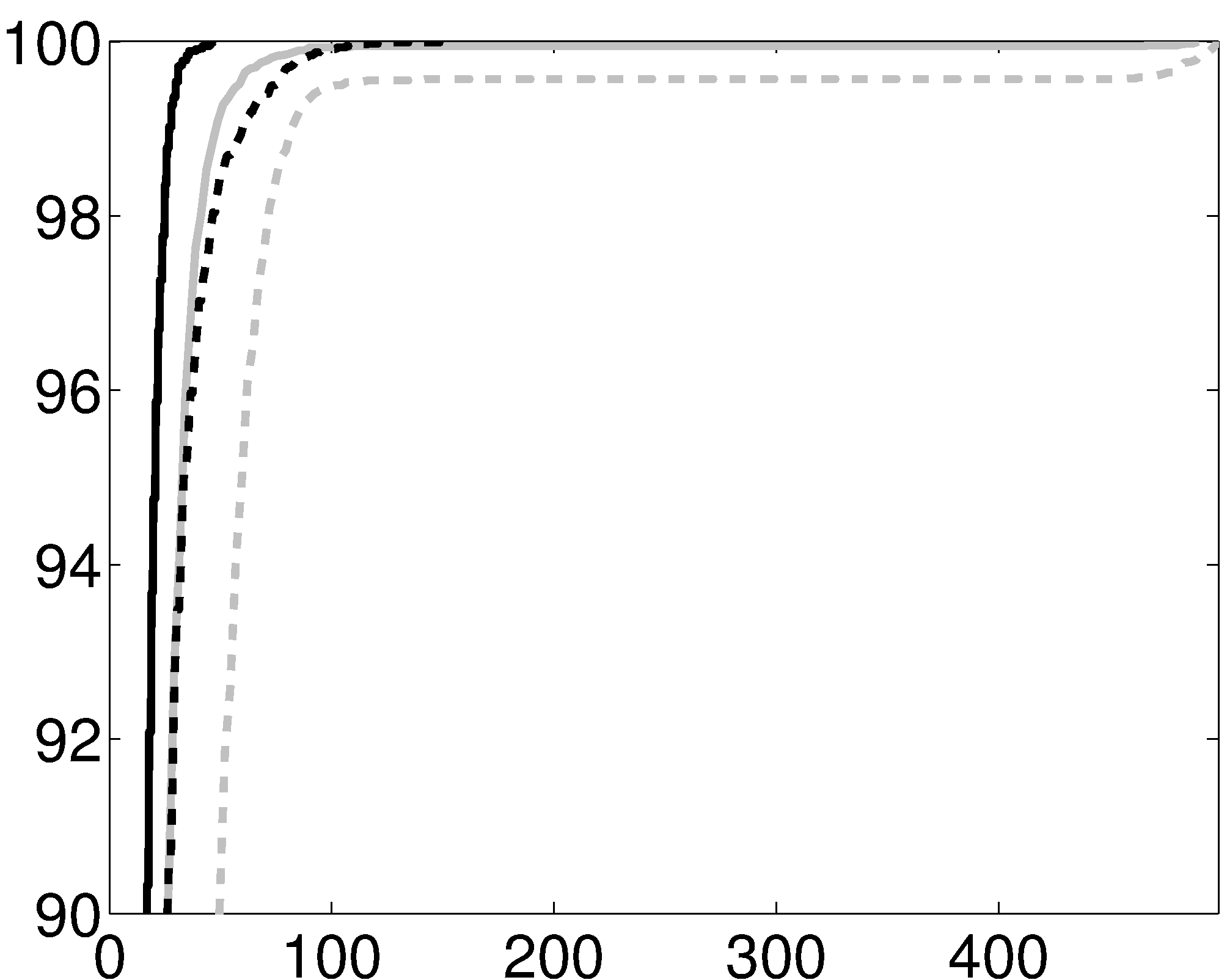}
\includegraphics[width=0.24\columnwidth]{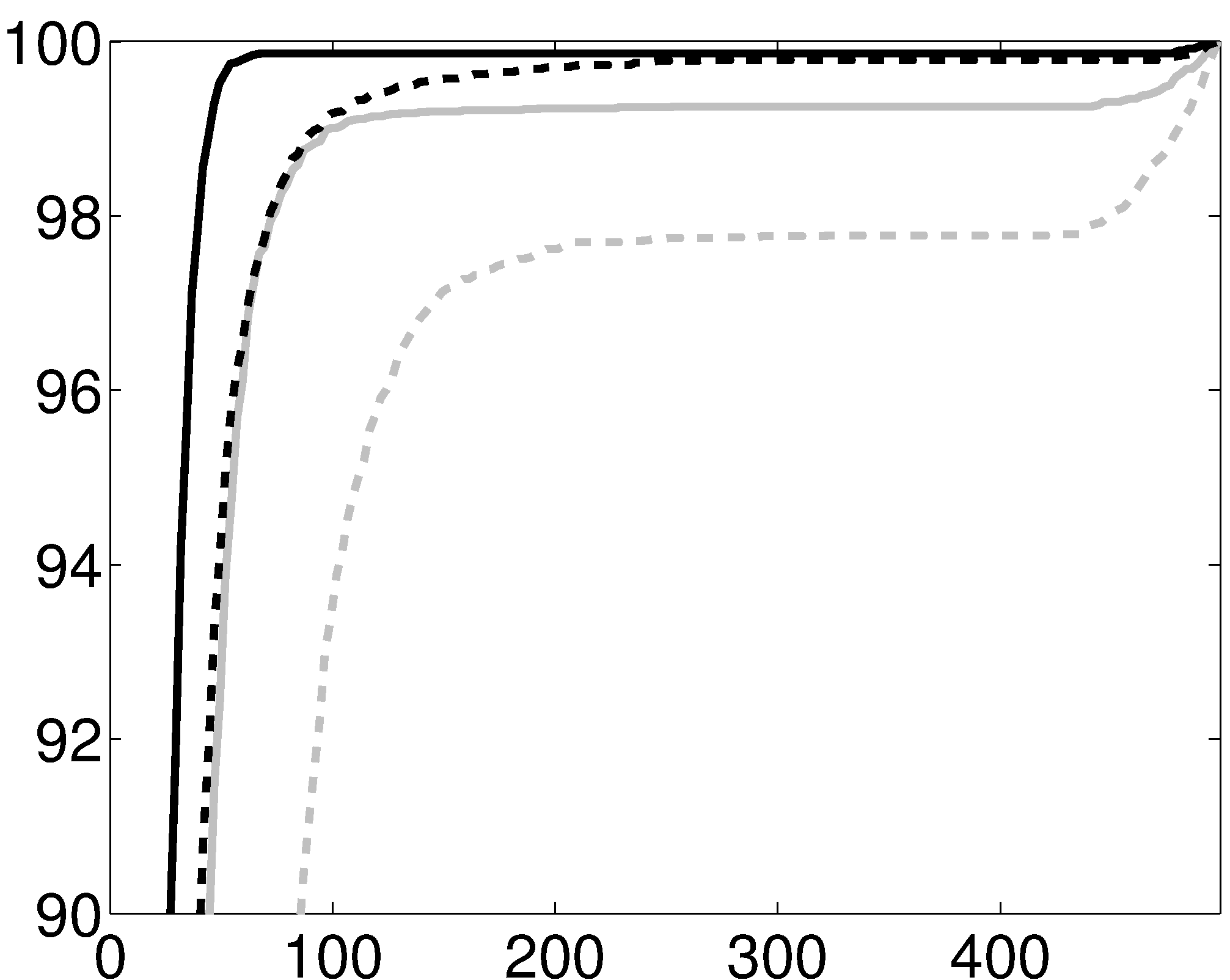}
\includegraphics[width=0.24\columnwidth]{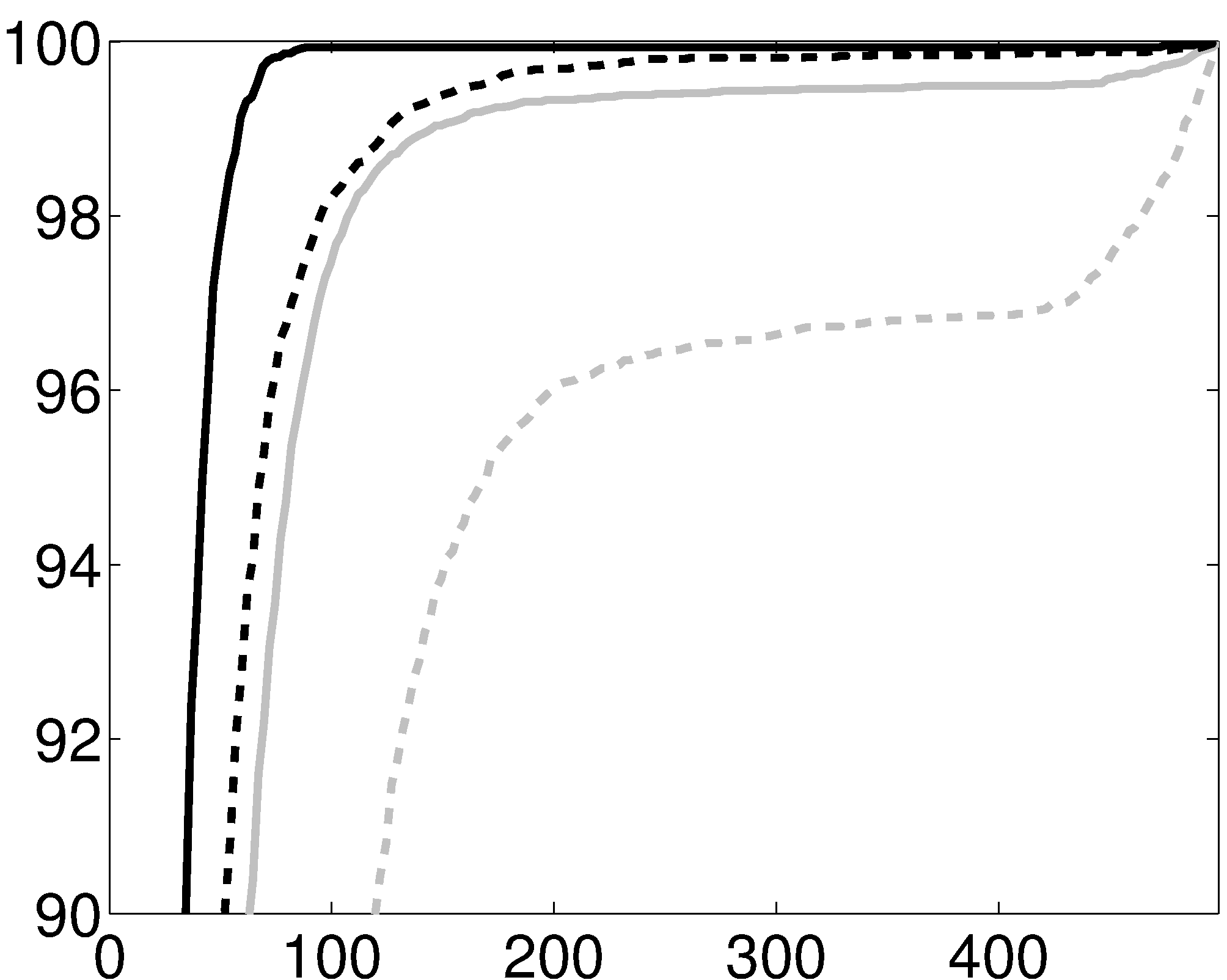}
\end{center}
\caption{\small The result of nearest neighbors estimation. In all subfigures, the x-axis is the true rank of an estimated nearest neighbor and the y-axis is its cumulative distribution. To emphasize the difference, we only show the area ranging from $90\%$ to $100\%$ in the y-axis. The gray dashed (gray, black dashed and black respectively) curve is the cumulative distribution of the true ranks of the estimated nearest neighbors estimated from the ordinary Euclidean distance (diffusion distance based on $L^{\text{NN}}(\widetilde{W},\widetilde{G})$, $L(\widetilde{W},\widetilde{G})$ and $L_0(\widetilde{W},\widetilde{G})$ respectively). From left to right: clean samples from the bell shaped manifold, noisy samples with $\alpha=1/4$ and $c=0.25,0.4,0.5$ respectively. It is clear that when the noise is large, the result based on the $L^2$ distance is much worse than the others. The result based on $L^{\text{NN}}(\widetilde{W},\widetilde{G})$ is slightly better, but not that ideal, $L(\widetilde{W},\widetilde{G})$ is even better and $L_0(\widetilde{W},\widetilde{G})$ is the best.
}\label{figure:numerical1:nnEst}
\end{figure}

\subsection{$2$-dim images}\label{Section:numerical:2dim}

In Subsection \ref{Section:numerical:1dim}, we investigated numerically the influence of noise on CGL methods when the connection function is trivial. In this subsection, we discuss an example where the connection function plays an essential role in the analysis. We consider a dataset which contains randomly rotated versions of a set of objects, and the task is to align these objects in addition to classifying them. We encounter this kind of datasets and problems in, for example, image processing \cite{singer_zhao_shkolnisky_hadani:2011,singer_wu:2012,Zhao_Singer:2013}, shape analysis \cite{Huang_Su_Guibas:2013}, phase retrieval problems \cite{Marchesini_Tu_Wu:2014,Alexeev_Bandeira_Fickus_Mixon:2013}, etc. In \cite{singer_zhao_shkolnisky_hadani:2011,singer_wu:2012,Zhao_Singer:2013,Marchesini_Tu_Wu:2014,Alexeev_Bandeira_Fickus_Mixon:2013} and others, the CGL methods have been applied to solve the problem.

To focus specifically on demonstrating the influence of noise on this problem, we work with 2-dimensional images observed in polar coordinates. If an image is defined in Cartesian coordinates, then in general a numerical rotation will introduce numerical artifacts and errors since resampling or interpolation procedures are then involved to compare two rotated images. These numerical issues are alleviated if we work with polar coordinates. To further minimize these numerical artifacts, we use surrogate images as our dataset -- by a {\it surrogate image}, we mean a function defined on the circle $S^1$, which is discretized into $p$ equally spaced points. In other words, we consider images defined in polar coordinates, where we only have one sample on the radial axis.

Now we discretize the $2\times 2$ rotational group, $SO(2)$, which is the same as the circle $S^1,$ into $p$ equally spaced points, that is,  ${\mathcal T}^{(2)}:=\{e^{i2\pi k/p}\}_{k=1}^p$ - the sample $e^{i2\pi k/p}\in {\mathcal T}^{(2)}$ simply rotates vectors in $\mathbb{R}^2$ by an angle $2\pi k/p$. Note that since the surrogate images are defined on $p$ equally spaced points on $S^1$, the rotations in ${\mathcal T}^{(2)}$ act exactly on the images without introducing any numerical error. 
We choose $n_K$ different surrogate images, denoted as $\{f_i\}_{i=1}^{n_K}\subset \RR^{p}$. Then we randomly and independently rotate each of them by $n_R$ angles; that is, for all $k=1,\ldots,n_K$ and $l=1,\ldots,n_R$, we have $S_i:=R_{k,l}\circ f_k$, where $R_{k,l}\in{\mathcal T}^{(2)}$, $R_{k,l}\circ f_k$ means rotating $f_k$ by $R_{k,l}$ and $i=(k-1)n_R+l$. We assume that $\argmin_{R\in \mathcal{T}^{(2)}}\|f_i-R\circ f_j\|>0$, for all $i,j=1,\ldots,n_K$; that is, the image $f_i$ is not a rotated version of another one $f_j$. In the end we get $n=n_Kn_R$ randomly rotated images $\{S_{i}\}_{i=1}^n\subset\RR^{p}$. Denote by $\sigma$ the standard deviation of all pixels of all images in $\{f_i\}_{i=1}^{n_K}$. The data $\mathcal{X}=\{I_i\}_{i=1}^n$ we analyze is the clean data contaminated by the noise which is i.i.d. sampled from $Z$, that is, we have $I_i=S_i+Z_i$. 

We now build up the connection graph by setting $\graphV:=\{I_i\}_{i=1}^n$ and $\graphE:=\{(i,j);\,I_i,I_j\in \graphV\}$; that is, we take the complete graph scheme. Next we assign the weight function as $w:(i,j)\mapsto e^{-d^2_{\text{RID}}(I_i,I_j)/m}$, where $m$ is the $25\%$ quantile of all non-zero RID distances defined on $\graphE$, and the connection function as $\relationship:(i,j)\mapsto \argmin_{R\in \mathcal{T}^{(2)}}\|I_i-R\circ I_j\|$. For comparison purposes, we also take the nearest neighbor scheme to construct the connection graph, denoted by $(\graphG^{\text{NN}},w^{\text{NN}},\relationship^{\text{NN}})$, where we choose $100$ nearest neighbors - as defined by the RID distance -  to construct edges. When the images are noise-free, due to the connection function, we can recover $R_{k,l}$'s up to a rotation from the top eigenvector $v_1$ of different CGL's built up from different connection graphs, $(\graphG,w,\relationship)$ or $(\graphG^{\text{NN}},w^{\text{NN}},\relationship^{\text{NN}})$, with or without removing the diagonal entries. To simplify the notation, we will use the same notation $v_1$ to denote the top eigenvector of the different CGL's. Precisely, the estimated rotation is built up from $v_1$, denoted as $v\in\CC^{n}$, by setting $v(i)=\frac{v_1(i)}{|v_1(i)|}$ when $|v_1(i)|>0$ and $v(i)=1$ when $|v_1(i)|=0$. (In a slight departure from the descriptions given earlier in the paper, the $r_{i,j}$'s are not $2\times 2$ matrices here, but complex numbers describing the corresponding rotations. Hence, $v_1$ is in $\CC^n$. If we had used $2\times 2$ matrices, $v_1$ would have been in $\RR^{2n}$ and we could have computed the vector $v$ by using pairs of entries of $v_1$.)

To evaluate the performance of the estimated rotation when noise exists, we construct a complex vector $u\in\CC^{n}$ whose $i$-th entry - where $i=(k-1)n_R+l$,  $k=1,\ldots,n_K$ and $l=1,\ldots,n_R$ -  is the complex form of $R_{k,l}$. We then evaluate the difference between the estimated rotation of the $l$-th object and the ground truth by observing the angle of $u(i)^*v(i)$. In other words, this quantity shows the discrepancy between the true rotation and the estimated rotation.
To visualize this result, we plot the vector $z\in \RR^{n}$ where $z(i)$ is the angle of the complex number $u(i)^*v(i)$. In Figure \ref{fig:numerical2}, the resulting $z$'s with $p=1000$, $n_K=5$, $n_R=200$, $\alpha=1/4$ and $c=6\sigma$ are illustrated. Note that since there are $5$ different images, we see a piecewise function with $5$ different values when the images are clean. When noise exists, we can see clearly the benefit of removing the diagonal entries (see Figure \ref{fig:numerical2}, (H)). 

We mention that depending on the problem, the affinity function and the connection function are constructed in different ways (see, for example \cite{singer_zhao_shkolnisky_hadani:2011,singer_wu:2012,Zhao_Singer:2013,Huang_Su_Guibas:2013,Marchesini_Tu_Wu:2014,Alexeev_Bandeira_Fickus_Mixon:2013}). 
Also, the CGL is only one of several techniques we could use to analyze datasets described by the connection graph. We might consider other techniques, such as, semidefinite programming relaxation \cite{Wang_Singer:2013}, nonlinear independent component analysis \cite{Talmon_Cohen_Gannot_Coifman:2013}, and other methods, to obtain the information we are interested in, reorganize the data, etc...

\begin{figure}
\begin{center}
\includegraphics[width=0.45\columnwidth]{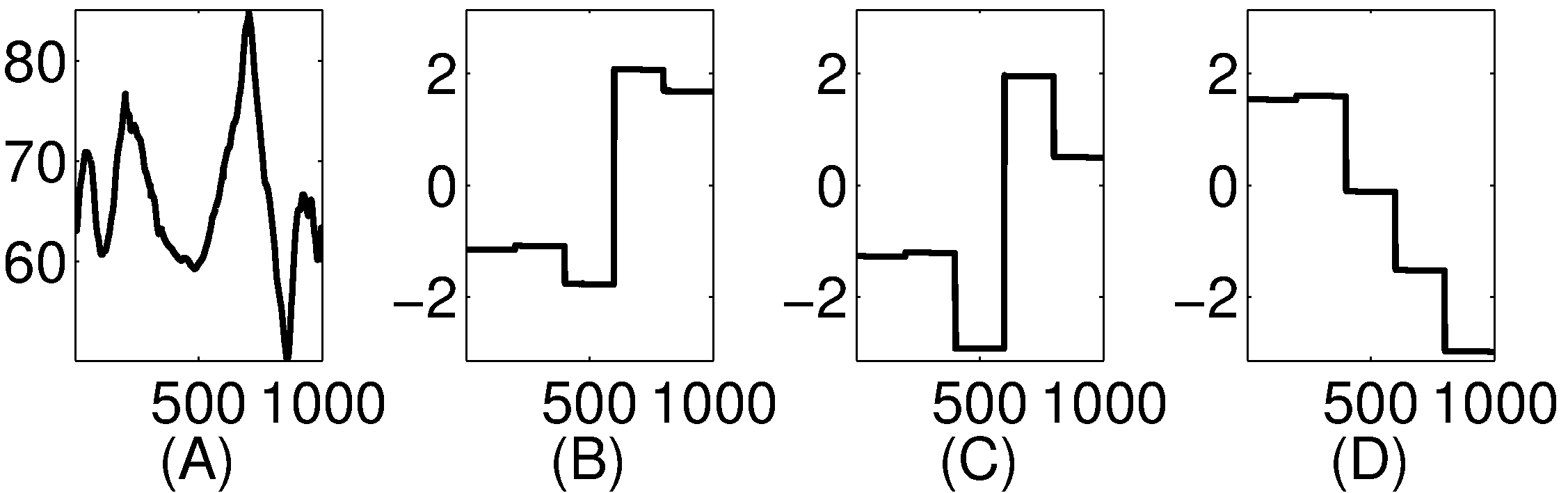}
\includegraphics[width=0.45\columnwidth]{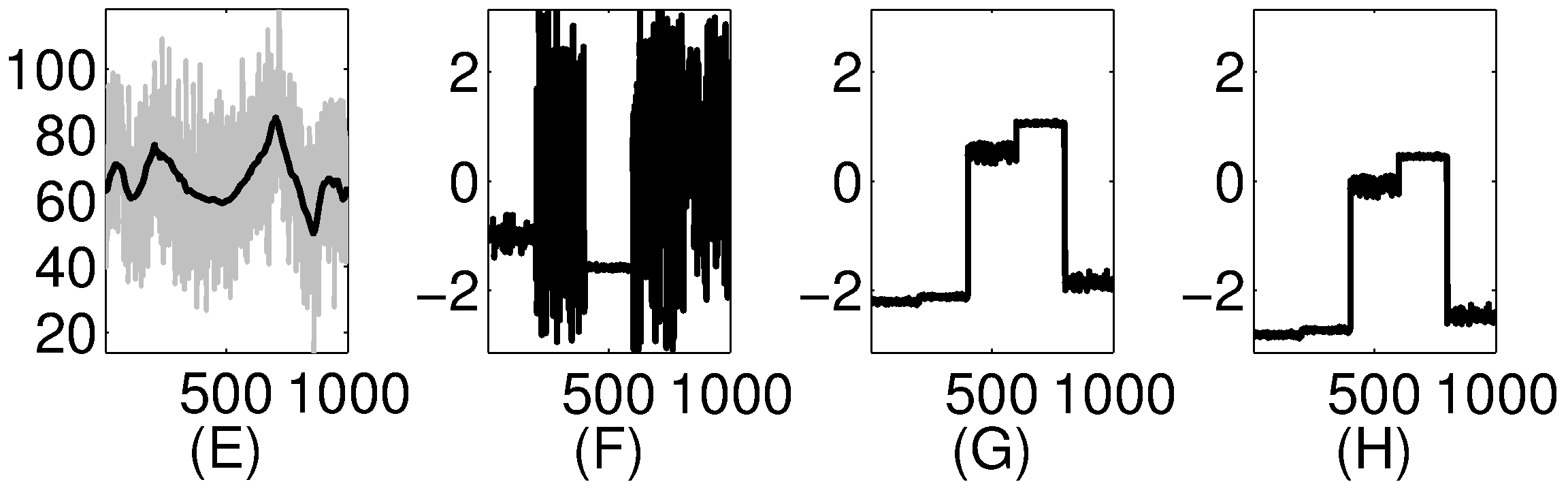}
\end{center}
\caption{\small (A): a clean surrogate image. (B)-(D): alignment vectors $z$ computed from clean images; (E): a noisy surrogate image. (F)-(H): alignment vectors $z$ computed from noisy images with $c=6\sigma$. (A) and (E): the black curve is a clean surrogate image, and the gray curve is its noisy version; (B) and (F): the result from the CGL built up from $\graphG^{\text{NN}}$; (C) and (G): the result from the CGL built from $\graphG$ and the diagonal entries are not removed; (D) and (H): the result from the CGL built from $\graphG$ with the diagonal entries removed. It is clear that when the images are clean, all different CGL's give equivalent results. But in the presence of noise, the CGL built up from $\graphG^{\text{NN}}$ is obviously worse.}\label{fig:numerical2}
\end{figure}

\bibliographystyle{plain}
\bibliography{research,noisyManifold} 

\def\cprime{$'$}
\begin{thebibliography}{10}

\bibitem{Al-Aifari_Daubechies_Lipman:2013}
R.~Al-Aifari, I.~Daubechies, and Y.~Lipman.
\newblock {Continuous Procrustes Distance Between Two Surfaces}.
\newblock {\em Comm. Pure Appl. Math.}, 66:934--964, 2013.

\bibitem{Alexeev_Bandeira_Fickus_Mixon:2013}
B.~Alexeev, A.~S. Bandeira, M.~Fickus, and D.~G. Mixon.
\newblock Phase retrieval with polarization.
\newblock {\em SIAM J. Imaging Sci.}, 2013.

\bibitem{Bandeira_Singer_Spielman:2013}
A.~S. Bandeira, A.~Singer, and D.~A. Spielman.
\newblock {A Cheeger Inequality for the Graph Connection Laplacian}.
\newblock {\em SIAM Journal on Matrix Analysis and Applications}, to appear,
  2013.
\newblock arXiv:1204.3873 [math.SP].

\bibitem{Batard_Sochen:2012}
T.~Batard and N.~Sochen.
\newblock {Polyakov action on ($\backslash$rho, G) -equivariant functions
  Application to color image regularization}.
\newblock In {\em Scale Space and Variational Methods in Computer Vision},
  pages 483--494. 2012.

\bibitem{Batard_Sochen:2014}
T.~Batard and N.~Sochen.
\newblock A class of generalized laplacians on vector bundles devoted to
  multi-channel image processing.
\newblock {\em J. Math. Imaging Vis.}, 48(3):517--543, 2014.

\bibitem{Bates:2014}
J.~Bates.
\newblock The embedding dimension of laplacian eigenfunction maps.
\newblock {\em Appl. Comput. Harmon. Anal.}, (0):--, 2014.

\bibitem{belkin_niyogi:2003}
M.~Belkin and P.~Niyogi.
\newblock {Laplacian Eigenmaps for Dimensionality Reduction and Data
  Representation}.
\newblock {\em Neural. Comput.}, 15(6):1373--1396, June 2003.

\bibitem{belkin_niyogi:2005}
M.~Belkin and P.~Niyogi.
\newblock Towards a theoretical foundation for {L}aplacian-based manifold
  methods.
\newblock In {\em Proceedings of the 18th Conference on Learning Theory
  (COLT)}, pages 486--500, 2005.

\bibitem{belkin_niyogi:2007}
M.~Belkin and P.~Niyogi.
\newblock {Convergence of laplacian eigenmaps}.
\newblock In {\em Advances in Neural Information Processing Systems 19:
  Proceedings of the 2006 Conference}, volume~19, page 129. The MIT Press,
  2007.

\bibitem{berard2}
P.~B\'{e}rard.
\newblock {\em Spectral Geometry: Direct and Inverse Problems}.
\newblock Springer, 1986.

\bibitem{berard_besson_gallot:1994}
P.~B\'erard, G.~Besson, and S.~Gallot.
\newblock Embedding riemannian manifolds by their heat kernel.
\newblock {\em Geom. Funct. Anal.}, 4:373--398, 1994.
\newblock 10.1007/BF01896401.

\bibitem{Berline_Getzler_Vergne:2004}
N.~Berline, E.~Getzler, and M.~Vergne.
\newblock {\em Heat Kernels and Dirac Operators}.
\newblock Springer, 2004.

\bibitem{bhatia97}
Rajendra Bhatia.
\newblock {\em Matrix analysis}, volume 169 of {\em Graduate Texts in
  Mathematics}.
\newblock Springer-Verlag, New York, 1997.

\bibitem{Bishop_Crittenden:2001}
R.~L. Bishop and R.~J. Crittenden.
\newblock {\em Geometry of Manifolds}.
\newblock Amer Mathematical Society, 2001.

\bibitem{Boyer_Lipman_StClair_Puente_Patel_Funkhouser_Jernvall_Daubechies:2011}
D.~M. Boyer, Y.~Lipman, E.~St.~Clair, J.~Puente, B.~A. Patel, T.~Funkhouser,
  J.~Jernvall, and I.~Daubechies.
\newblock Algorithms to automatically quantify the geometric similarity of
  anatomical surfaces.
\newblock {\em P. Natl. Acad. Sci. USA}, 108(45):18221--18226, 2011.

\bibitem{Chen_Lin_Chern:2013}
P.~Chen, C.~Lin, and I.~Chern.
\newblock A perfect match condition for point-set matching problems using the
  optimal mass transport approach.
\newblock {\em SIAM J. Imaging Sci.}, 6(2):730--764, 2013.

\bibitem{cheng_wu:2012}
M.-Y. Cheng and H.-T. Wu.
\newblock Local linear regression on manifolds and its geometric
  interpretation.
\newblock {\em J. Am. Stat. Assoc.}, 108:1421--1434, 2013.

\bibitem{Fan:1996}
F.~Chung.
\newblock {\em Spectral Graph Theory}.
\newblock American Mathematical Society, 1996.

\bibitem{Chung_Kempton:2013}
F.~Chung and M.~Kempton.
\newblock A local clustering algorithm for connection graphs.
\newblock In Anthony Bonato, Michael Mitzenmacher, and Pawel Pralat, editors,
  {\em Algorithms and Models for the Web Graph}, volume 8305 of {\em Lecture
  Notes in Computer Science}, pages 26--43. Springer International Publishing,
  2013.

\bibitem{Chung_Zhao_Kempton:2013}
F.~Chung, W.~Zhao, and M.~Kempton.
\newblock Ranking and sparsifying a connection graph.
\newblock In Anthony Bonato and Jeannette Janssen, editors, {\em Algorithms and
  Models for the Web Graph}, volume 7323 of {\em Lecture Notes in Computer
  Science}, pages 66--77. Springer Berlin Heidelberg, 2012.

\bibitem{coifman_lafon:2006}
R.~R. Coifman and S.~Lafon.
\newblock Diffusion maps.
\newblock {\em Appl. Comput. Harmon. Anal.}, 21(1):5--30, 2006.

\bibitem{Coifman_Maggioni:2006}
R.~R. Coifman and M.~Maggioni.
\newblock Diffusion wavelets.
\newblock {\em Appl. Comput. Harmon. Anal.}, 21(1):53 -- 94, 2006.

\bibitem{Collins_Zomorodian_Carlsson_Guibas:2004}
A.~Collins, A.~Zomorodian, G.~Carlsson, and L.~J. Guibas.
\newblock A barcode shape descriptor for curve point cloud data.
\newblock {\em Computers \& Graphics}, 28(6):881 -- 894, 2004.

\bibitem{Cucuringu_Lipman_Singer:2012}
M.~Cucuringu, Y.~Lipman, and A.~Singer.
\newblock Sensor network localization by eigenvector synchronization over the
  euclidean group.
\newblock {\em ACM Transactions on Sensor Networks}, 8(3):19:1--19:42, 2012.

\bibitem{Cucuringu_Singer_Cowburn:2012}
M.~Cucuringu, A.~Singer, and D.~Cowburn.
\newblock Eigenvector synchronization, graph rigidity and the molecule problem.
\newblock {\em Information and Inference: A Journal of the IMA}, 1:21--67,
  2012.

\bibitem{DiaconisFreedmanProjPursuit84}
Persi Diaconis and David Freedman.
\newblock Asymptotics of graphical projection pursuit.
\newblock {\em Ann. Statist.}, 12(3):793--815, 1984.

\bibitem{doCarmo:1992}
M.P. do~Carmo and F.~Flaherty.
\newblock {\em Riemannian Geometry}.
\newblock Birkhauser Boston, 1992.

\bibitem{nekSparseMatrices}
N.~{El Karoui}.
\newblock Operator norm consistent estimation of large dimensional sparse
  covariance matrices.
\newblock {\em The Annals of Statistics}, 36(6):2717--2756, December 2008.

\bibitem{nekCorrEllipD}
N.~{El Karoui}.
\newblock Concentration of measure and spectra of random matrices: Applications
  to correlation matrices, elliptical distributions and beyond.
\newblock {\em The Annals of Applied Probability}, 19(6):2362--2405, December
  2009.

\bibitem{nekInfoPlusNoiseKernelMatrices10}
N.~{El Karoui}.
\newblock On information plus noise kernel random matrices.
\newblock {\em Ann. Statist.}, 38(5):3191--3216, 2010.

\bibitem{ElKaroui_Wu:2013}
N.~{El Karoui} and H.-T. {Wu}.
\newblock {Vector diffusion maps and random matrices with random blocks}.
\newblock {\em ArXiv e-prints}, 2013.
\newblock arXiv:1310.0188 [math.PR].

\bibitem{NEKRobustPaperPNAS2013Published}
Noureddine El~Karoui, Derek Bean, Peter~J. Bickel, Chinghway Lim, and Bin Yu.
\newblock On robust regression with high-dimensional predictors.
\newblock {\em Proceedings of the National Academy of Sciences}, 2013.

\bibitem{Epstein:2007}
C.~Epstein.
\newblock {\em Introduction to the Mathematics of Medical Imaging}.
\newblock SIAM, second edition, 2007.

\bibitem{comparingTopKLists2003}
R.~Fagin, R.~Kumar, and D.~Sivakumar.
\newblock Comparing top k lists.
\newblock {\em SIAM Journal on Discrete Mathematics}, 17(1):134--160, 2003.

\bibitem{Giannakis_Schwander_Ourmazd:2012}
D.~Giannakis, P.~Schwander, and A.~Ourmazd.
\newblock The symmetries of image formation by scattering. i. theoretical
  framework.
\newblock {\em Opt. Express}, 20(12):12799--12826, Jun 2012.

\bibitem{gilkey:1974}
P.~Gilkey.
\newblock {\em The Index Theorem and the Heat Equation}.
\newblock Princeton, 1974.

\bibitem{Hadani_Singer:2011b}
R.~Hadani and A.~Singer.
\newblock {Representation theoretic patterns in three dimensional Cryo-Electron
  Microscopy I: The intrinsic reconstitution algorithm}.
\newblock {\em Annals of Mathematics}, 174(2):1219--1241, 2011.

\bibitem{HallMarronNeemanJRSSb05}
P.~Hall, J.~S. Marron, and A.~Neeman.
\newblock Geometric representation of high dimension, low sample size data.
\newblock {\em J. R. Stat. Soc. Ser. B Stat. Methodol.}, 67(3):427--444, 2005.

\bibitem{hein_audibert_luxburg:2005}
M.~Hein, J.~Audibert, and U.~von Luxburg.
\newblock From graphs to manifolds - weak and strong pointwise consistency of
  graph {L}aplacians.
\newblock In {\em Proceedings of the 18th Conference on Learning Theory
  (COLT)}, pages 470--485, 2005.

\bibitem{Huang_Su_Guibas:2013}
Q.-X. Huang, H.~Su, and L.~Guibas.
\newblock Fine-grained semi-supervised labeling of large shape collections.
\newblock {\em ACM Transactions on Graphics (TOG)}, 32(6):190, 2013.

\bibitem{Jones_Maggioni_Schul:2008}
P.~W. Jones, M.~Maggioni, and R.~Schul.
\newblock {Manifold parametrizations by eigenfunctions of the Laplacian and
  heat kernels.}
\newblock {\em P. Natl. Acad. Sci. USA}, 105(6):1803--8, February 2008.

\bibitem{LaurentMassart2000}
B.~Laurent and P.~Massart.
\newblock Adaptive estimation of a quadratic functional by model selection.
\newblock {\em Ann. Statist.}, 28(5):1302--1338, 2000.

\bibitem{ledoux2001}
M.~Ledoux.
\newblock {\em The concentration of measure phenomenon}, volume~89 of {\em
  Mathematical Surveys and Monographs}.
\newblock American Mathematical Society, Providence, RI, 2001.

\bibitem{Lee_OveisGharan_Trevisan:2012}
R.~Lee, J, S.~Oveis~Gharan, and L.~Trevisan.
\newblock Multi-way spectral partitioning and higher-order cheeger
  inequalities.
\newblock In {\em Proceedings of the Forty-fourth Annual ACM Symposium on
  Theory of Computing}, STOC '12, pages 1117--1130, 2012.

\bibitem{Marchesini_Tu_Wu:2014}
S.~{Marchesini}, Y.-C. {Tu}, and H.-T. {Wu}.
\newblock {Alternating Projection, Ptychographic Imaging and Phase
  Synchronization}.
\newblock {\em ArXiv e-prints}, 2014.
\newblock arXiv:1402.0550 [math.OC].

\bibitem{Memoli:2011}
F.~M\'emoli.
\newblock A spectral notion of gromov-wasserstein distance and related methods.
\newblock {\em Appl. Comput. Harmon. Anal.}, 30(3):363 -- 401, 2011.

\bibitem{Memoli_Sapiro:2005}
F.~M\'emoli and G.~Sapiro.
\newblock {A Theoretical and Computational Framework for Isometry Invariant
  Recognition of Point Cloud Data}.
\newblock {\em Found. Comput. Math.}, 5:313--347, 2005.

\bibitem{nadler_lafon_coifman:2005}
B.~Nadler, S.~Lafon, R.~R. Coifman, and I.~G. Kevrekidis.
\newblock Diffusion maps, spectral clustering and eigenfunctions of
  {Fokker-Planck} operators.
\newblock In Y.~Weiss, B.~Sch\"{o}lkopf, and J.~Platt, editors, {\em Adv. Neur.
  In.}, volume~18, pages 955--962, Cambridge, MA, 2006. MIT Press.

\bibitem{nadler_lafon_coifman:2006}
B.~Nadler, S.~Lafon, R.~R. Coifman, and I.~G. Kevrekidis.
\newblock Diffusion maps, spectral clustering and reaction coordinates of
  dynamical systems.
\newblock {\em Appl. Comput. Harmon. Anal.}, 21(1):113--127, 2006.

\bibitem{Niyogi_Smale_Weinberger:2009}
P.~Niyogi, S.~Smale, and S.~Weinberger.
\newblock Finding the homology of submanifolds with high confidence from random
  samples.
\newblock In {\em Twentieth Anniversary Volume:}, pages 1--23. Springer New
  York, 2009.

\bibitem{Ovsjanikov_Sun_Guibas:2008}
M.~Ovsjanikov, J.~Sun, and L.~Guibas.
\newblock Global intrinsic symmetries of shapes.
\newblock In {\em Proceedings of the Symposium on Geometry Processing}, SGP
  '08, pages 1341--1348. Eurographics Association, 2008.

\bibitem{Portegies:2013}
J.~W. {Portegies}.
\newblock {Embeddings of Riemannian manifolds with heat kernels and
  eigenfunctions}.
\newblock {\em ArXiv e-prints}, 2013.
\newblock arXiv:1311.7568 [math.DG].

\bibitem{Qiu_Hancock:2007}
H.~Qiu and E.R. Hancock.
\newblock Clustering and embedding using commute times.
\newblock {\em Pattern Analysis and Machine Intelligence, IEEE Transactions
  on}, 29(11):1873--1890, Nov 2007.

\bibitem{Reuter:2010}
M.~Reuter.
\newblock {Hierarchical Shape Segmentation and Registration via Topological
  Features of Laplace-Beltrami Eigenfunctions}.
\newblock {\em Int. J. Comput. Vision}, 89:287--308, 2010.

\bibitem{Rustamov:2007}
R.~M. Rustamov.
\newblock Laplace-beltrami eigenfunctions for deformation invariant shape
  representation.
\newblock In {\em Proceedings of the fifth Eurographics symposium on Geometry
  processing}, SGP '07, pages 225--233. Eurographics Association, 2007.

\bibitem{singer:2006}
A.~Singer.
\newblock From graph to manifold {Laplacian}: The convergence rate.
\newblock {\em Appl. Comput. Harmon. Anal.}, 21(1):128--134, 2006.

\bibitem{singer:2011}
A.~Singer.
\newblock {Angular Synchronization by Eigenvectors and Semidefinite
  Programming.}
\newblock {\em Appl. Comput. Harmon. Anal.}, 30(1):20--36, 2011.

\bibitem{coifman_singer:2008}
A.~Singer and R.~R. Coifman.
\newblock Non-linear independent component analysis with diffusion maps.
\newblock {\em Appl. Comput. Harmon. Anal.}, 25(2):226 -- 239, 2008.

\bibitem{singer_wu:2012}
A.~Singer and H.-T. Wu.
\newblock Vector diffusion maps and the connection {Laplacian}.
\newblock {\em Comm. Pure Appl. Math.}, 65(8):1067--1144, 2012.

\bibitem{singer_wu:2013a}
A.~Singer and H.-T. Wu.
\newblock 2-d tomography from noisy projections taken at unknown random
  directions.
\newblock {\em SIAM J. Imaging Sci.}, 6(1):136--175, 2013.

\bibitem{singer_wu:2013}
A.~Singer and H.-T. Wu.
\newblock Spectral convergence of the connection laplacian from random samples.
\newblock {\em submitted}, 2013.

\bibitem{singer_zhao_shkolnisky_hadani:2011}
A.~Singer, Zhao Z., Shkolnisky Y., and Hadani R.
\newblock Viewing angle classification of cryo-electron microscopy images using
  eigenvectors.
\newblock {\em SIAM J. Imaging Sci.}, 4(2):723--759, 2011.

\bibitem{stewart90}
G.~W. Stewart and Ji~Guang Sun.
\newblock {\em Matrix perturbation theory}.
\newblock Computer Science and Scientific Computing. Academic Press Inc.,
  Boston, MA, 1990.

\bibitem{Sun_Ovsjanikov_Guibas:2009}
J.~Sun, M.~Ovsjanikov, and L.~Guibas.
\newblock A concise and provably informative multi-scale signature based on
  heat diffusion.
\newblock In {\em Proceedings of the Symposium on Geometry Processing}, SGP
  '09, pages 1383--1392. Eurographics Association, 2009.

\bibitem{Szlam_Maggioni_Coifman:2008}
A.~D. Szlam, M.~Maggioni, and R.~R. Coifman.
\newblock {Regularization on Graphs with Function-adapted Diffusion Processes}.
\newblock {\em The Journal of Machine Learning Research}, 9:1711--1739, 2008.

\bibitem{Talmon_Cohen_Gannot_Coifman:2013}
R.~Talmon, I.~Cohen, S.~Gannot, and R.~Coifman.
\newblock Diffusion maps for signal processing: A deeper look at
  manifold-learning techniques based on kernels and graphs.
\newblock {\em Signal Processing Magazine, IEEE}, 30(4):75--86, July 2013.

\bibitem{VonLuxburg_Belkin_Bousquet:2008}
U.~von Luxburg, M.~Belkin, and O.~Bousquet.
\newblock {Consistency of spectral clustering}.
\newblock {\em Ann. Stat.}, 36(2):555--586, 2008.

\bibitem{Wang_Huang_Guibas:2013}
F.~Wang, Q.~Huang, and L.~J. Guibas.
\newblock Image co-segmentation via consistent functional maps.
\newblock In {\em The IEEE International Conference on Computer Vision (ICCV)},
  December 2013.

\bibitem{Wang_Singer:2013}
L.~Wang and A.~Singer.
\newblock {Exact and Stable Recovery of Rotations for Robust Synchronization}.
\newblock {\em Information and Inference: A Journal of the IMA}, 2013.
\newblock accepted for publication.

\bibitem{wu:2012}
H.-T. Wu.
\newblock Embedding riemannian manifolds by the heat kernel of the connection
  laplacian.
\newblock {\em submitted}, 2012.

\bibitem{Zhao_Singer:2013}
Z.~{Zhao} and A.~{Singer}.
\newblock {Rotationally Invariant Image Representation for Viewing Direction
  Classification in Cryo-EM}.
\newblock {\em ArXiv e-prints}, 2013.
\newblock arXiv:1309.7643 [math.CV].

\end{thebibliography}

\newpage

	\appendix
	\begin{center}
	\textbf{\textsc{APPENDIX}}\\
	to ``Connection graph Laplacian methods can be made robust to noise''
	\end{center}
	\renewcommand{\thesection}{\Alph{section}}
	\renewcommand{\theequation}{\Alph{section}-\arabic{equation}}
	\renewcommand{\thelemma}{\Alph{section}-\arabic{lemma}}
	\renewcommand{\thecorollary}{\Alph{section}-\arabic{corollary}}
	\renewcommand{\thesubsection}{\Alph{section}-\arabic{subsection}}
	\renewcommand{\thesubsubsection}{\Alph{section}-\arabic{subsection}.\arabic{subsubsection}}
	\setcounter{equation}{0}  
	\setcounter{lemma}{0}
	\setcounter{corollary}{0}
	\setcounter{section}{0}
	
\section{Technical results}
\subsection{On quadratic forms}\label{subsec:ConcQuadForms}
\begin{lemma}\label{lemma:controlSupQuadForms}
Suppose $Z_1,\ldots,Z_n$ are random vectors in $\mathbb{R}^p$, with $Z_i=\Sigma_i^{1/2} X_i$, where $X_i$ has mean 0 and covariance $\id_p$. We further assume that for every convex 1-Lipschitz function $f$, if $m_{f(X_i)}$ is a median of $f(X_i)$, $P(|f(X_i)-m_{f(X_i)}|>t)\leq 2 \exp(-c_it^2)$. $Z_i$'s are possibly dependent. Let $\{Q_i\}_{i=1}^n$ be $p\times p$ positive definite matrices. Call $\opnorm{Q_i}$ the largest eigenvalue of $Q_i$. Then we have 
$$
\sup_{1\leq i \leq n}\left|\sqrt{Z_i\trsp Q_i Z_i}-\Exp{\sqrt{Z_i\trsp Q_i Z_i}}\right|=\bO_P(\sup_i \sqrt{\opnorm{Q_i\Sigma_i/c_i}} \sqrt{\log n})\;.
$$ 
This implies that, when $\sup_i \sqrt{\opnorm{Q_i\Sigma_i/c_i}} \sqrt{\log n}\tendsto 0$\;,
$$
\sup_{1\leq i \leq n}\left|Z_i\trsp Q_i Z_i-\trace{\Sigma_i Q_i}\right|=\bO_P(\sup_i \sqrt{\opnorm{Q_i\Sigma_i/c_i}}\sqrt{\log n} \left[\sup_i \sqrt{\trace{\Sigma_i Q_i}}\vee 1\right])\;.
$$
\end{lemma}

As explained in \cite{ledoux2001}, the condition we require on $X_i$ is satisfied by many distributions. We refer also to \cite{nekCorrEllipD} for many examples. Here are two examples. The Gaussian distribution in dimension $p$ satisfies the previous assumptions with $c_i=1/2$, independently of the dimension. When $X_i$'s have independent coordinates supported on intervals of width at most $B_i$, $c_i$ is proportional to $1/B_i$. 

\begin{proof}
The map $X_i\rightarrow \sqrt{X_i\trsp\Sigma_i^{1/2}Q_i\Sigma_i^{1/2} X_i}$ is convex and $\sqrt{\opnorm{Q_i\Sigma_i}}=\sqrt{\opnorm{\Sigma_i^{1/2}Q_i\Sigma_i^{1/2}}}$-Lipschitz as a function of $X_i$. Indeed, it is a norm, which gives convexity. The Lipschitz-property comes from the triangle inequality. Hence, under our assumptions, since $Z_i=\Sigma_i^{1/2}X_i$, we have 
$$
P\left(\left|\sqrt{Z_i\trsp Q_i Z_i}-\Exp{\sqrt{Z_i\trsp Q_i Z_i}}\right|>t\right)\leq 2\exp(-c_i t^2/(\opnorm{Q_i\Sigma_i}))\;.
$$	

By a simple union bound, we get 
$$
P\left(\sup_{1\leq i \leq n}\left|\sqrt{Z_i\trsp Q_i Z_i}-\Exp{\sqrt{Z_i\trsp Q_i Z_i}}\right|>t\right)\leq 2\sum_{i=1}^n \exp(-c_i t^2/[\opnorm{Q_i\Sigma_i}])\leq 
2 n \exp(- t^2/(\sup_i \opnorm{Q_i\Sigma_i/c_i}))
$$
Taking $t_K=K\sqrt{\log(n)\opnorm{Q_i\Sigma_i/c_i}}$, for $K$ a constant,  gives the first result. 
The second result follows from remarking that $|a^2-b^2|=|a-b||a+b|\leq |a-b|^2+2|b||a-b|$. When $\sup_i \sqrt{\opnorm{Q_i\Sigma_i/c_i}} \sqrt{\log n}\tendsto 0$, this gives immediately  
$$
\sup_{1\leq i \leq n}\left|Z_i\trsp Q_i Z_i-\left[\Exp{\sqrt{Z_i\trsp Q_i Z_i}}\right]^2\right|=\bO_P(\sup_i \sqrt{\opnorm{Q_i\Sigma_i/c_i}}\sqrt{\log n} \sup_i \left[\sqrt{\trace{\Sigma_i Q_i}}\vee 1\right])\;, 
$$
after we notice that $\trace{\Sigma_iQ_i}=\Exp{Z_i\trsp Q_i Z_i}\geq \left[\Exp{\sqrt{Z_i\trsp Q_i Z_i}}\right]^2$.
Finally, using the variance bound in Proposition 1.9 of \cite{ledoux2001}, we see that, 
$$
\Exp{Z_i\trsp Q_i Z_i}-\left[\Exp{\sqrt{Z_i\trsp Q_i Z_i}}\right]^2\leq 2\opnorm{\Sigma_i Q_i/c_i}\;.
$$
Under our assumption that $\sup_i \sqrt{\opnorm{Q_i\Sigma_i/c_i}} \sqrt{\log n}\tendsto 0$, we have $\sup_i \sqrt{\opnorm{Q_i\Sigma_i/c_i}}\tendsto 0$ and therefore 
$$
\sup_i \opnorm{\Sigma_i Q_i/c_i} = \lo(\sup_i\sqrt{\opnorm{\Sigma_i Q_i/c_i}})\;.
$$
This gives the second bound. 
\end{proof}

In the case of the Gaussian distribution, the previous bounds can be improved, using an observation found in \cite{LaurentMassart2000}.

\begin{lemma}\label{lemma:controlSupQuadFormsLaurentMassart}
Suppose $Z_1,\ldots,Z_n$ are random vectors in $\mathbb{R}^p$, with $Z_i\sim{\cal N}(0,\Sigma_i)$. $Z_i$'s are possibly dependent. Let $\{Q_i\}_{i=1}^n$ be $p\times p$ positive definite matrices.  
Then we have, if $S_i=\Sigma_i^{1/2}Q_i\Sigma_i^{1/2}$
$$
\sup_{1\leq i \leq n}\left|Z_i\trsp Q_i Z_i-\trace{S_i}\right|=\bO_P\left(\sup_{1\leq i \leq n}\sqrt{\log(n)}\sqrt{\trace{S_i^2}}+\opnorm{S_i}\log(n)\right)\;.
$$
\end{lemma}

\begin{proof}
By rotational invariance of the Gaussian distribution, we have
$$
W_i\triangleq Z_i\trsp Q_i Z_i-\trace{S_i}\equalInLaw \sum_{k=1}^p \lambda_k(S_i) (X_k^2-1)\;,
$$
where $X_k$'s are i.i.d ${\cal N}(0,1)$. 
Using Lemma 1, p, 1325 in \cite{LaurentMassart2000}, we see that 
$$
P\left(\frac{|W_i|}{2}>\sqrt{\trace{S_i^2}}\sqrt{x}+\opnorm{S_i} x\right)\leq \exp(-x)\;.
$$
Taking $x=K\log(n)$ in the previous inequality and a simple union bound gives the announced result. 
\end{proof}

\subsection{Proof of Lemma \ref{lemma:approxBoundsLaplacian}} \label{app:sec:ApproxResLaplMatrix}

\begin{proof}
We have $L(W,G)=D^{-1}S=(D/n)^{-1}(S/n)$. If we call $d_{i,i}=\sum_{j\neq i} w_{i,j}$ and $\tilde{d}_{i,i}=\sum_{j\neq i} \tilde{w}_{i,j}$, we see that 
$$
|d_{i,i}/n-\tilde{d}_{i,i}/n|\leq \sup_{j\neq i} |w_{i,j}-\tilde{w}_{i,j}|\;.
$$
Hence, 
$$
\sup_{1\leq i \leq n} |d_{i,i}/n-\tilde{d}_{i,i}/n|\leq \sup_{1\leq i \leq n} \sup_{j\neq i} |w_{i,j}-\tilde{w}_{i,j}|\leq \sup_{i,j}  |w_{i,j}-\tilde{w}_{i,j}|\leq \eps\;.
$$
We conclude that 
$$
\opnorm{D/n-\widetilde{D}/n}\leq \eps\;.
$$
Under our assumptions, it is clear that $\opnorm{(D/n)^{-1}}\leq 1/\gamma$. The previous display also implies that $\opnorm{(\widetilde{D}/n)^{-1}}\leq 1/(\gamma-\eps)$. 

Furthermore, since 
$$
(D/n)^{-1}-(\widetilde{D}/n)^{-1}=(D/n)^{-1}[D/n-\widetilde{D}/n](\widetilde{D}/n)^{-1}\;,
$$
we see that 
$$
\opnorm{(D/n)^{-1}-(\widetilde{D}/n)^{-1}}\leq \frac{\eps}{\gamma (\gamma-\eps)}\;.
$$

Also, 
$$
\norm{S/n-\widetilde{S}/n}_F^2\leq \sup_{i,j}\norm{S_{i,j}-\widetilde{S}_{i,j}}_F^2\;.
$$
Naturally, since $S_{i,j}=w_{i,j}G_{i,j}$ and $\widetilde{S}_{i,j}=\widetilde{w}_{i,j}\widetilde{G}_{i,j}$,
$$
\norm{S_{i,j}-\widetilde{S}_{i,j}}_F^2\leq |w_{i,j}|^2 \norm{G_{i,j}-\widetilde{G}_{i,j}}_F^2+|w_{i,j}-\widetilde{w}_{i,j}|^2 \norm{\widetilde{G}_{i,j}}_F^2\leq C^2 (\eta^2+\eps^2)\;.
$$
Hence, 
$$
\opnorm{S/n-\widetilde{S}/n}\leq \norm{S/n-\widetilde{S}/n}_F\leq C (\eta+\eps).
$$
We also note that $\norm{\widetilde{S}/n}_F\leq C^2$. So we can conclude that 
$$
\opnorm{D^{-1}S-\widetilde{D}^{-1}\widetilde{S}}\leq \opnorm{D^{-1}(S-\widetilde{S})+(D^{-1}-\widetilde{D}^{-1})\widetilde{S}}\leq \frac{1}{\gamma} C(\eta+\eps)+\frac{\eps}{\gamma (\gamma-\eps)} C^2\;.
$$
	
\end{proof}

\subsection{$\card{\setOfTransforms}$: an example when $\setOfTransformskDims_\text{exact}\subset SO(k)$}\label{subsec:cardExactRotations}

Naturally, when working with discretized images/objects with $p$ pixels/voxels, we need to also discretize $SO(k)$. In light of results like Proposition \ref{prop:controlApproxdijNoisy}, one natural question we have to deal with concerns the cardinality of the discretized set of transformations, $\setOfTransformskDims$, and the corresponding set for companion matrices, $\setOfTransforms$. The following proposition answers this question.

The images/objects are discretized in polar coordinates. In other words, each point on our grid can be identified by its location on a ray emanating from the origin and reaching a point $p$ on the sphere of radius $r_0$ centered at the origin. The discretization of each ray does not have to be uniform. But this discretization is the same for all rays. 

In the case of $SO(2)$, this simply means that we discretize the circle of radius $r_0$, and our points lay on the corresponding radii. In this situation, it is natural to represent each point on our grid through $(r,\theta)$. To give a concrete example, we assume that $\theta \in \{2\pi\frac{k}{M}\}_{k=0}^{M-1}$ and $r\in \{\mathsf{r}_1,\ldots,\mathsf{r}_\alpha\}$ with $\mathsf{r}_\alpha=r_0$. The discretization of $SO(2)$ corresponds simply to rotations by an angle $\theta_k$, where $\theta_k=2\pi\frac{k}{M}$. These rotations clearly map our grid onto itself.

We assume that our images or objects, after having been uniformly discretized in polar coordinates, fit in a $k$-dimensional cube. Then we assume that the rotation group is properly discretized so that each rotation is exact in the sense that it commutes with discretization. In other words, the rotation does not change the pixel values  - pixels are simply swapped  and pixel values are not averaged or aggregated in other ways. Note that when the image or object is discretized in  Cartesian coordinates, the discretization and rotation will not commute and a distortion is inevitable. We postpone the study of such a phenomenon to future work.  

\begin{proposition} \label{prop:cardinalSetOfDiscretizedRotations}
When $\setOfTransformskDims_\text{exact}$ is the discretized version of $SO(k)$ we just discussed, we have 
$$
\card{\setOfTransformskDims_\text{exact}}=\card{{\cal T}_\text{exact}}=\gO(p^{k-1})\;.
$$
Furthermore, the elements of ${\cal T}_\text{exact}$ are permutation matrices. In particular, they are orthogonal matrices. 
\end{proposition}

\textbf{Comment :} Proposition \ref{prop:cardinalSetOfDiscretizedRotations} shows that in checking Assumption \textbf{G1}, we can assume that $\card{\setOfTransforms}$ is polynomial in $p$. This implies that Assumption G1 will be satisfied when $\max(\sigma_p,\sqrt{p}s_p^2) =\lo([\log(np)]^{-1/2})$. Hence the conditions we will have to check on $\sigma_p$ and $\sqrt{p}s_p^2$ will be quite unrestrictive and we will see that this implies that CGL algorithms are robust to considerable amount of additive noise. 

\begin{proof} 
Our polar-coordinate discretization amounts to discretizing a sphere of radius $r_0$ in $\mathbb{R}^k$ with $M$ points and discretizing each ray linking a point on that sphere to the origin along $\alpha$ points. We naturally have the relationship $M\alpha=p$. 

Now elements of $\setOfTransformskDims_\text{exact}$ are orthogonal matrices with determinant 1, hence they can be characterized by their action on $k-1$ vectors in $\mathbb{R}^k$ which span a subspace of dimension $k-1$. 

Let us pick $k-1$ elements among our $M$ points on the sphere of radius $r_0$. We require that these $k-1$ elements span a subspace of dimension $k-1$ in $\mathbb{R}^k$. We call the corresponding vectors $v_1,\ldots,v_{k-1}$. 

Suppose now that $\transformkDims \in \setOfTransformskDims_\text{exact}$. Then, for each $i$, $\transformkDims v_i$ has to be one of the elements of our discretized sphere. Therefore, 
$$
\card{\setOfTransformskDims_\text{exact}}\leq M^{k-1}=\left(\frac{p}{\alpha}\right)^{k-1}\leq p^{k-1}\;.
$$

Now let $\transformkDims \in \setOfTransformskDims_\text{exact}$ and let $O$ be the companion matrix of $\transformkDims$. Note that if $I_i$ is our image and $I_i^\vee \in \mathbb{R}^p$ is its discretized version,  $OI_i^\vee$ swaps the position of the entries of the vector $I_i^\vee$, since $\transformkDims$ maps our grid onto itself. Hence $O$ is a permutation matrix and it is therefore orthogonal.   
\end{proof}

\paragraph{Another approach} We note that another approach can be employed to generate a sampling grid $\mathfrak{X}$ and an associated set of exact transformations for $k\geq 3$. Take $m$ points in $R^k$  denoted as $X=\{x_i\}_{i=1}^m$. Take a finite subgroup, $T$ of $SO(k)$. 
Now consider the sampling grid $\mathfrak{X}\triangleq \{Rx_i;\, R\in T, x_i\in X\}$. Since $T$ is a subgroup, we know that the sampling grid $\mathfrak{X}$ is of finite size; that is, $|\mathfrak{X}|\leq m|T|$. It is also clear that $T$ is an exact set of transforms for $\mathfrak{X}$, by simply using the fact that $T$ is a group. 

Note that the standard polar coordinate grid in $\mathbb{R}^2$ described above can be viewed as an instance of the method we just discussed, with $T$ consisting of powers of the rotation by the angle $\frac{2\pi}{M}$. 

Note however that the classification of finite subgroups of $SO(k)$, for $k\geq 3$, imposes strong constraints on the sampling grids obtained by such a construction.

\newtheorem{nekDef}{Definition}
\section{Background on CGL methods}\label{Appendix:Section:CGL}

In this section, we discuss the {\it noise-free} connection graph Laplacian (CGL) $\vdmC$ defined in (\ref{definition:LWG}), understand its asymptotical behavior under the assumption that the point clouds we collect are distributed on a manifold, and show that the CGL matrix built up under this assumption enjoy a sparsity property which allows the robustness result shown in this paper. In addition, we will discuss the fact that the CGL matrix can be viewed as a generalization of the graph Laplacian (GL) \cite{singer_wu:2013}. We will see that although GL and CGL share several similar properties but are fundamentally different. 

We would assume the background knowledge of differential geometry in the following discussion. For a reader who is not familiar with the subject, we refer him to \cite{doCarmo:1992,gilkey:1974,Bishop_Crittenden:2001,berard2,Berline_Getzler_Vergne:2004} for the topics we will encounter. 

We start from some notations. Denote $\manifold$ to be a $d$-dimensional compact, connected and smooth Riemannian manifold embedded in $\RR^p$ via $\embedding$, where $d\leq p$. Denote the tangent bundle as $T\manifold$. The tangent plane at $y\in\manifold$ is denoted as $T_y \manifold$. Introduce the metric $g$ on $\manifold$ induced from the canonical metric of the ambient space $\RR^p$. Denote $d(y,y')$ to be the geodesic distance between $y,y'\in\manifold$.
Denote by $\nabla$ the covariant derivative of the vector field, $\Delta_g$ the Laplace-Beltrami operator, $\nabla^2$ the connection Laplacian of the tangent bundle associated with the Levi-Civita connection, and by $\Ric$ the Ricci curvature of $(\manifold,g)$. 
We denote the spectrum of $\nabla^2$ (resp. $\Delta_g$) by $\{-\lambda_l\}_{l=0}^\infty$ (resp. $\{-\gamma_l\}_{l=0}^\infty$), where $0=\lambda_0\leq\lambda_1\leq \ldots$ (resp. $0=\gamma_0<\gamma_1\leq \ldots$), and the corresponding eigenspaces by $F_l:=\{\VectorFieldOnM\in L^2(T\manifold):~\nabla^2\VectorFieldOnM=-\lambda_l \VectorFieldOnM\}$ (resp. $E_l:=\{\phi\in L^2(\manifold):~\Delta_g\phi=-\gamma_l \phi\}$), $l=0,1,\ldots$. In general, while $\gamma_0=0$ always exists, $\lambda_0$ may not: a simple example is found considering $S^2$ with the standard metric.
It is well known \cite{gilkey:1974} that $\dim(F_l)<\infty$, the eigen-vector-fields are smooth and form a basis for $L^2(T\manifold)$ (resp. $\dim(E_l)<\infty$, the eigenfunctions are smooth and form a basis for $L^2(\manifold)$), that is,
$L^2(T\manifold)=\overline{\oplus_{l\in\NN\cup\{0\}} F_l}$ (resp. $L^2(\manifold)=\overline{\oplus_{l\in\NN\cup\{0\}} E_l}$), the completion of $\oplus_{l\in\NN\cup\{0\}} F_l$ with relative to the measure induced by $g$. To simplify the statement, we assume that $\lambda_l$ (resp. $\gamma_l$) for each $l$ are simple and $\VectorFieldOnM_l$ (resp. $\phi_l$) is a normalized basis of $F_l$ (resp. $E_l$). Denote $\mathcal{B}(F_k)$ (resp. $\mathcal{B}(E_k)$) the set of bases of $F_k$ (resp. $E_k$), which is identical to the orthogonal group $O(\dim(F_k))$ (resp. $O(\dim(E_k))$). Denote the set of the corresponding orthonormal bases of $L^2(T\manifold)$ by
$\mathcal{B}(T\manifold,g)=\Pi^\infty_{k=1}\mathcal{B}(F_k)$ (resp. $\mathcal{B}(\manifold,g)=\Pi^\infty_{k=1}\mathcal{B}(E_k)$).

Given the collected data $\dataset=\{x_i\}_{i=1}^n\subset \RR^p$, where $x_i$ are {\it signal random vector} i.i.d. sampled from a random vector $\mathtt{X}$. 
We assume a manifold structure inside the signal random vector; that is, we view $\mathtt{X}:\Omega \rightarrow \RR^p$ as a measurable function with respect to the probability space $(\Omega,\mathcal{F},P)$, and assume that its range is $\embedding(\manifold)$. Note that we cannot define the probability density function (p.d.f.) of $\mathtt{X}$ on the ambient space $\RR^p$ when $d<p$ since $\iota(\manifold)$ is degenerate, but we can still discuss how we sample points from $\iota(\manifold)$, which leads to {\it the p.d.f. of $\mathtt{X}$ on $\manifold$}. Indeed, we employ the following definition which is based on the induced measure \cite[Section 4]{cheng_wu:2012}.
Denote $\tilde{\mathcal{B}}$ to be the Borel sigma algebra on $\embedding(\manifold)$, and $\tilde{P}_{\mathtt{X}}$ the probability measure of $\mathtt{X}$, defined on $ \tilde{\mathcal{B}}$, induced from $P$. Assume that $\tilde{P}_{\mathtt{X}}$ is absolutely continuous with respect to the volume density on $\embedding(\manifold)$ associated with $g$, that is, $\ud \tilde{P}_{\mathtt{X}}(x)=f(\embedding^{-1}({x}))\embedding_*\ud V(x)$, where $f:\manifold\to \RR$ and $x\in\embedding(\manifold)$. 
We interpret $f$ as {\it the p.d.f. of $\mathtt{X}$ on $\manifold$}.
To alleviate the notation, in the following we abuse the notation and will not distinguish between $\embedding(\manifold)$ and $\manifold$. 

Pick up the kernel function $K(x)=e^{-x^2}$. Note that the kernel function can be more general, for example, $K\in C^2(\RR)$, non-zero, non-negative and monotonic decreasing, but we focus ourselves on this kernel to make the explanation clear.
Denote $ \mu^{(k)}_{l}:=\int_{\mathbb{R}^d}\|x\|^l K^{(k)}(\|x\|)\ud x$, where $k=0,1,2$, $l\in \NN\cup\{0\}$, and $K^{(k)}$ means the $k$-th order derivative of $K$. We assume $\mu^{(0)}_0=1$ and $\sqrt{h}$ is small enough so that $\sqrt{h}$ is smaller than the reach \cite{Niyogi_Smale_Weinberger:2009} and the injectivity radius \cite{doCarmo:1992} of the manifold $\manifold$, $\text{inj}(\manifold)$.

\subsection{Connection Graph and Affinity Graph}\label{Appendix:B_2}

An affinity graph, denoted as $(\graphG,w)$, is actually a special case of the connection graph in the sense that the connection function is not defined on the affinity graph. 
If we take a constant function $\relationship_0:\graphE\to 1$, the affinity graph becomes a connection graph $(\graphG,w,\relationship_0)$.  

We mention that in practice, if we decide to construct the connection graph from a given dataset, there are several different ways depending on the application and goal. For example, in addition to the examples discussed in the main context, in the geometric approach to the signal processing \cite{coifman_singer:2008,Talmon_Cohen_Gannot_Coifman:2013}, the affinity is defined to be the Mahalanobis distance reflecting the intrinsic property of the underlying state space; in the chair synchronization problem \cite{Huang_Su_Guibas:2013}, the affinity between two chair meshes is defined based on their Hausdorff distance. The quality of the chosen affinity might influence the analysis result directly.

\subsection{Connection Graph Laplacian and its Applications}
Now we discuss the CGL. Consider the symmetric matrix $\vdmC_s := \vdmD^{-1/2}\vdmS\vdmD^{-1/2}$, which is similar to $\vdmC$. Since $\vdmC_s$ is symmetric, it has a complete set of eigenvectors $v_{n,i}$, $i=1,\ldots,nd$ and its associated eigenvalues $\mu_{n,i}$, where the eigenvalues are bounded by $1$ \cite{singer_wu:2012}. We would order the eigenvalues in the decreasing order. Note that the eigenvectors of $\vdmC_s$ is related to those of $\vdmC$ via $\vdmD^{-1/2}$. 

First, note that $\vdmC$ is an operator acting on $\boldsymbol{v}\in \RR^{nd}$ by
\begin{align}\label{Cmatrix}
(\vdmC\boldsymbol{v})[i] = \frac{\sum_{j: (i,j)\in E}w(i,j)r(i,j)\boldsymbol{v}[j]}{\sum_{k: (i,k)\in E} w(i,k)},
\end{align}
where $\boldsymbol{v}$ can be viewed as a vector-valued function defined on $\graphV$ so that $\boldsymbol{v}[j]:=(\boldsymbol{v}((j-1)d+1),\ldots,\boldsymbol{v}(jd))\in\RR^d$. 
We could interpret this formula as a generalized random walk on the affinity graph. Indeed, if we view the vector-valued function $\boldsymbol{v}$ as the status of a particle defined on the vertices, when we move from one vertex to the other one, the status is modified according to the relationship between vertices {\it encoded} in $\relationship$. We mention that depending on the connection function, the structure $\vdmC$ might be very different, which leads to different analysis results and conclusion. We will give a precise example regarding this statement later. 
Now we discuss an important property of the CGL -- the synchronization, which has been studied in \cite{Chung_Zhao_Kempton:2013,Bandeira_Singer_Spielman:2013} and applied to the following problems, for example,
\begin{enumerate}
\item a new imaging technique aiming to obtain the atomic scale resolution images of a macro-scale object called ``ptychographic imaging problem'' \cite{Marchesini_Tu_Wu:2014};\vspace{-.3cm}
\item a frame design called ``polarization'' for the phase retrieval problem \cite{Alexeev_Bandeira_Fickus_Mixon:2013};\vspace{-.3cm}
\item a spectral relaxation approach to solve the least squares solution of the rotational synchronization problem \cite{Wang_Singer:2013};\vspace{-.3cm}
\item graph realization problem by synchronization over the euclidean group \cite{Cucuringu_Lipman_Singer:2012,Cucuringu_Singer_Cowburn:2012}.
\end{enumerate} 
Here we give an intuition about this synchronization notion. Suppose there exists a vector-valued status $\boldsymbol{v}\in\RR^{nd}$ of norm $1$ which is ``synchronized'' according to the encoded relationship $\relationship$ in the sense that $\boldsymbol{v}[j]=\relationship(j,i)\boldsymbol{v}[i]$ for all $(i,j)\in \graphE$, then $\vdmC\boldsymbol{v}[i]$ will be the same as $\boldsymbol{v}[i]$, and hence the functional associated with the eigenvalue problem
$$
\max_{\boldsymbol{v}\in\RR^{nd};\,\|\boldsymbol{v}\|=1}\boldsymbol{v}^T\vdmC\boldsymbol{v}
$$ 
is maximized with the eigenvalue $1$, and its eigenvector is the synchronized vector $\boldsymbol{v}$. Thus, the top eigenvector of $\vdmC$, when viewed as the vector-valued status on the vertex, is the ``synchronized'' status with respect to the connection function. We mention that the existence of the synchronized vector-valued function is equivalent to the notion of ``consistency'' studied in \cite{Chung_Zhao_Kempton:2013}.

When the connection function is constant; that is, the connection matrix $G_0:=\boldsymbol{1}\boldsymbol{1}^T$, where $\boldsymbol{1}$ is a $n\times 1$ vector with all entries $1$, the GL is defined as $\dmL:=\dmI-L(W,G_0)$. Notice a natural interpretation of $L(W,G_0)$ -- since the sum of each row of $L(W,G_0)$ is $1$, $L(W,G_0)$ is the {\it transition matrix} associated with a random walk on $\graphG$. To avoid confusion, the eigenvectors and eigenvalues of $L(W,G_0)_s$ are denoted as $u_{n,i}$ and $\nu_{n,i}$, where $i=1,\ldots,n$, and $0\leq\nu_{n,i}\leq 1$ are ordered in the decreasing order.
As a special case of CGL, the GL has several applications which deserves discussion, for example 
\begin{enumerate}
\item in the spectral clustering algorithm, we only need to find the first $k$ trivial eigenvectors \cite{VonLuxburg_Belkin_Bousquet:2008,Lee_OveisGharan_Trevisan:2012}; \vspace{-.3cm}
\item to evaluate the Cheeger ratio, we need the second eigenvalue \cite{Fan:1996};\vspace{-.3cm}
\item to visualize the high dimensional data, we need the first $3$ non-trivial eigenvectors;\vspace{-.3cm}
\item in the cryo-EM problem, if we want to reconstruct the rotational position of each projection image, we need the first $9$ non-trivial eigenvectors \cite{Giannakis_Schwander_Ourmazd:2012}; \vspace{-.3cm}
\item in the 2d random tomography problem \cite{singer_wu:2013a}, only the first $2$ non-trivial eigenvectors are needed; \vspace{-.3cm}
\item in the orientability detection problem \cite{singer_wu:2012}, we need the first eigenvector. 
\end{enumerate}
More algorithms depending on the eigenstructure of the GL can be found, to mention but a few, in \cite{Coifman_Maggioni:2006,Szlam_Maggioni_Coifman:2008,Jones_Maggioni_Schul:2008,Sun_Ovsjanikov_Guibas:2009,Ovsjanikov_Sun_Guibas:2008,Rustamov:2007,Reuter:2010,Memoli:2011}. We comment that if the p.d.f $f$ is not uniform, then a specific normalization stated in \cite{coifman_lafon:2006} allows us to study the dynamics of the underlying dynamical system \cite{nadler_lafon_coifman:2005,nadler_lafon_coifman:2006}.

\subsection{Asymptotical behavior of CGL}
To better understand the CGL, we focus on the frame bundle and its associated tangent bundle \cite{singer_wu:2012} here to simplify the exploration. For CGL associated with a more general principal bundle structure, we refer the reader to \cite{singer_wu:2013}.

\begin{assumption}\label{Assumption:A}
\begin{enumerate}
\item[(D1)] $\manifold$ is a smooth and compact $d$-dim manifold. When the boundary $\partial\manifold$ is not empty, it is assumed to be smooth, and we denote $\manifold_\delta:=\{x\in\manifold:\,d(x,\partial\manifold)\leq \sqrt{\delta}\}$; 
\item[(D2)] The p.d.f. $\pdf\in C^3(\manifold)$ is uniformly bounded from below and above, that is, $0<p_m\leq \pdf(x)\leq p_M<\infty$. However, to simplify the exploration, we assume here that $\pdf$ is uniform, that is, $\pdf$ is a constant function defined on $\manifold$. When $\pdf$ is non-uniform, its theoretical results can be found in \cite{coifman_lafon:2006,singer_wu:2013}.
\end{enumerate}
\end{assumption}
Under Assumption \ref{Assumption:A}, we collect the data $\mathcal{X}$ independently and identically sampled from $\manifold$ and build up the following graph. First define a graph $\graphG_\manifold:=(\graphV,\graphE)$ by taking $\graphV=\mathcal{X}$, and $\graphE=\{(x_i,x_j);\,x_i\in\mathcal{X}\}$. Then define the affinity function $w$ on $\graphE$ by 
$$
w:\,(i,j)\mapsto K_h\left(x_i,x_j\right):=K\left( \|x_i-x_j \|^2_{\mathbb{R}^p}/h \right),
$$
where $h>0$ is the chosen bandwidth. Note that we choose to use the Euclidean distance, instead of the geodesic distance, to build up $w$ since in practice we have only an access to the Euclidean distance (or other metric, depending on the application). Asymptotically this discrepancy will disappear.
\begin{assumption}\label{Assumption:B}
\begin{enumerate}
\item[(D3)] For each point $x_i\in\mathcal{X}$, we also have a sample on the frame bundle $b(i)\in O(\manifold)$ so that the $b(i)$ is the basis of the tangent space $T_{x_i}\manifold$. In particular, we are given a group-valued function $b:\mathcal{X}\to O(d)$. 
\end{enumerate}
\end{assumption}

With Assumption \ref{Assumption:B}, we define the {\it connection function} on $\graphE$ as 
$$
\relationship:\,(i,j)\mapsto b(i)^{T}P_{x_i,x_j}b(j)\in O(d)
$$ 
where $P_{x_i,x_j}$ presents the parallel transport of the vector field from $x_j$ to $x_i$. 
As a result, we have a connection graph $(\graphG_\manifold, w,\relationship)$.
With $(\graphG_\manifold, w,\relationship)$, we build up the the CGL by $\id-L(W,G)$, where $W$ and $G$ are the weight matrix and connection matrix associated with $w$ and $\relationship$. Under this framework, the GL is when we work with the trivial line bundle associated with $\manifold$.

The geometrical meaning of the connection function deserves some discussions. First, note that although all tangent planes $T_{x_i}\manifold$ are isomorphic to $\RR^d$ \cite{doCarmo:1992}, but they are different in the sense that we cannot ``compare'' $T_{x_i}\manifold$ and $T_{x_j}\manifold$ directly. Precisely, it makes sense the say $u-v$ when $u,v\in\RR^d$, but we can not evaluate $u_i-u_j$ when $u_i\in T_{x_i}\manifold$ and $u_j\in T_{x_j}\manifold$. To carry out the comparison between different tangent planes, we need a bit more work. Indeed, $b(i)\in O(d)$ is a basis of the tangent plane of $T_{x_i}\manifold$, which practical meaning is mapping $\RR^d$ isomorphically to $T_{x_i}\manifold$. In other words, given a vector field $Y$, $b(i)^TY(x_i)$ evaluates its coordinate at $x_i$. The parallel transport $P_{x_i,x_j}$ is a geometrical generalization of the notion ``translation'' in the Euclidean space -- it is an isometric map mapping $T_{x_j}\manifold$ to $T_{x_i}\manifold$. As a result, $r(i,j)\in O(d)$ is an isometric map from $\RR^d$ to $\RR^d$, and geometrically it maps the coordinate of a vector field at $x_j$, that is, $\boldsymbol{v}[j]$, to the vector field at $x_i$, that is, $b(j)\boldsymbol{v}[j]$, then parallelly transports $b(j)\boldsymbol{v}[j]$ to $x_i$, and then evaluate the coordinate of $P_{x_i,x_j}b(j)\boldsymbol{v}[j]$ with related to the basis $b(i)$. 
We emphasize that the connection function in the connection graph associated with the frame bundle encodes not only the geometry but also the topology of the manifold. In practice, this constraint may lead to a better understanding of the underlying data structure. For example, in the cryo-EM problem, this viewpoint leads to a better angular classification result.

We now state the pointwise convergence and the spectral convergence of the $\vdmC$. These theorems apply to the GL, while we replace the vector fields by the functions with the same regularity and the connection Laplacian operator by the Laplace-Beltrami operator and $b(i)=1$ (see \cite{singer_wu:2013} for details).
\begin{thm}[CGL Pointwise Convergence \cite{singer_wu:2012,singer_wu:2013}] \label{thm:pointwise_conv_VDM}
Suppose Assumption \ref{Assumption:A} and Assumption \ref{Assumption:B} hold and $\VectorFieldOnM\in C^4(T\manifold)$. For all $x_i\notin \manifold_{\sqrt{h}}$ with high probability (w.h.p.)
\begin{align}
 b(i)\big((\vdmI-\vdmC)\VectorFieldOnMvec\big)[i]=h\frac{\mu_2}{2d}\nabla^2\VectorFieldOnM(x_i)+O(h^{2})+O\left(\frac{1}{n^{1/2}h^{d/4-1/2}}\right)\nonumber
\end{align} 
where $\VectorFieldOnMvec\in\RR^{nd}$ and $\VectorFieldOnMvec[i]=b(i)^{-1}\VectorFieldOnM(x_i)$. For all $x_i\in\manifold_{\sqrt{h}}$, we have w.h.p.
\begin{align}
b(i)\big((\vdmI-\vdmC)\VectorFieldOnMvec\big)[i]=O(\sqrt{h})P_{x_i,x_0}\nabla_{\partial_d}\VectorFieldOnM(x_0)+O(h)+O\left(\frac{1}{n^{1/2}h^{d/4-1/2}}\right),\nonumber
\end{align}
where $x_0=\argmin_{y\in\partial\manifold}d(x_i,y)$ and $\nabla_{\partial_d}$ is the derivative in the normal direction. 
\end{thm}
To state the spectral convergence result, define an operator $T_{\text{C},h}:C(T\manifold)\to C(T\manifold)$:
\[
T_{\text{C},h}\VectorFieldOnM(y):=\frac{\sum_{j: (i,j)\in E} K_h(y,x_j)P_{y,x_j}\VectorFieldOnM(x_j)}{\sum_{j: (i,j)\in E} K_h(y,x_j)},
\]
where $\VectorFieldOnM\in C(T\manifold)$. To simplify the discussion, we assume that the eigenvalues of the heat kernel of the connection Laplacian $e^{t\nabla^2}$ are simple. When there exists an eigenvalue with multiplicity greater than $2$, the theorem can be proved using the projection operators onto the eigenspaces. We mention that in the special case GL, the point convergence theorem was first established in \cite{belkin_niyogi:2005} under the uniform sampling and boundary-free assumption, and then extended to a more general setup in \cite{hein_audibert_luxburg:2005,coifman_lafon:2006,singer:2006,singer_wu:2013}, and the spectral convergence was established in \cite{belkin_niyogi:2007,VonLuxburg_Belkin_Bousquet:2008}.
\begin{thm}[CGL Spectral Convergence \cite{singer_wu:2013}]\label{thm:spectralconvergence_heatkernel}
Suppose Assumption \ref{Assumption:A} and Assumption \ref{Assumption:B} hold and fix $t>0$. Denote $\mu_{\text{C},t,h,i}$ to be the $i$-th eigenvalue of $T^{t/h}_{\text{C},h}$
with the associated eigenvector $\VectorFieldOnM_{\text{C},t,h,i}$. Also denote $\mu_{t,i}>0$ to be the $i$-th eigenvalue of $e^{t\nabla^2}$ with the associated eigen-vector field $\VectorFieldOnM_{t,i}$. We assume that both $\mu_{\text{C},t,h,i}$ and $\mu_{t,i}$ decrease as $i$ increases, respecting the multiplicity. Fix $i\in\NN$. Then there exists a sequence $h_n\to 0$ such that 
$$
\lim_{n\to \infty}\mu_{\text{C},t,h_n,i}=\mu_{t,i},\,\,\mbox{ and }\,\,\lim_{n\to \infty}\|\VectorFieldOnM_{\text{C},t,h_n,i}-\VectorFieldOnM_{t,i}\|_{L^2(T\manifold)}=0
$$ 
in probability.
\end{thm}

With these Theorems, we are able to discuss why different connection functions lead to different analysis results. Consider $S^2$ embedded in $\RR^3$ with the standard metric. If we define the connection function according to the Levi-Civita connection, then the top eigenvalue of $e^{t\nabla^2}$ is strictly less than $1$ due to the hairy-ball theorem \cite{gilkey:1974}. In other words, asymptotically we are not able to find a synchronized vector-valued status on it. On the other hand, if we take the trivial connection function, that is, $\relationship(i,j)=I_2$, then asymptotically we obtain $\Delta_g$ acting on two independent functions. Since the dimension of the null space of $\Delta_g$ is the number of the connected components of the manifold, the top eigenvalue of the CGL with the trivial connection function is $1$;  that is, a synchronized vector-valued status exists. See Figure \ref{S.B:fig:2} for the result.

The main reason leading to this difference is rooted in the connection theory, and we refer the interested reader to \cite{Bishop_Crittenden:2001}. 

\begin{figure}
\begin{center}
\includegraphics[width=0.6\columnwidth]{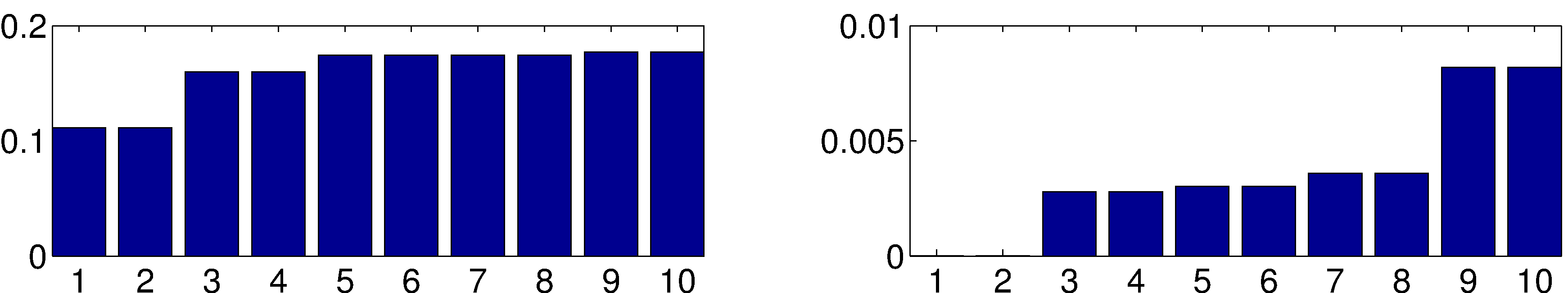}
\end{center}
\caption{\small The first $10$ eigenvalues of the $\vdmI-\vdmC$ with related to a non-trivial connection function determined from the Levi-Civita connection of the frame bundle of $S^2$ (left) and a trivial connection function (right). $1000$ points are sampled uniformly from $S^2$. Note that the eigenvalues on the right figure are the same as those of GL while the multiplicities of all the eigenvalues are $2$.}\label{S.B:fig:2}
\end{figure}

\subsection{The ``Sparsity'' of the CGL}

We define the following ``sparsity'' condition.
\begin{nekDef}
Fix $\gamma>0$. For a $n\times n$ matrix $Q$, we sort its eigenvalues $\nu_{Q,\ell}$, $\ell=1,\ldots,n$, so that $|\nu_{Q,1}|\geq |\nu_{Q,2}|\geq \ldots\geq |\nu_{Q,n}|$. Then $Q$ satisfies the {\it $\gamma$-sparsity property} if
$$
|\nu_{Q,\ell}|\leq e^{-C_Q\ell^{\gamma}}
$$ 
for all $\ell$, where $C_Q>0$ depends on $Q$.
\end{nekDef}
We now claim that the CGL under the manifold assumption satisfies the $2/d$-sparsity property. Note that this theorem also holds for the GL.
\begin{thm}
Asymptotically when $n\to\infty$, for $\ell\in\NN$, $\mu_{n,\ell}\leq e^{-C_{\text{L}}\ell^{2/d}}$,  
where the constants $C_{\text{L}}>0$ depends $d$, the lower bound of the Ricci curvature $k$ and the diameter $D$.
\end{thm}
\begin{proof}
 Note that the Weyl's theorem \cite{Berline_Getzler_Vergne:2004} holds for the connection Laplacian, that is,
\[
\tilde{N}(\mu)\sim \frac{1}{(4\pi)^{d/2}\Gamma(d/2+1)}\mu^{d/2},
\] 
where $\tilde{N}(\mu)$ is the number of eigenvalues of $\nabla^2$ less than $\mu>0$, and we have the consequence that \cite{wu:2012}
\begin{align} 
\lambda_j\geq c'(d,k,D)j^{2/d},
\end{align}
where $j\in\NN$ and $c'(d,k,D)$ is the universal constant depending only on $d$, the lower bound of the Ricci curvature $k$ and the diameter $D$. Note that since $\mu_{n,\ell}\to e^{-\lambda_\ell}$ in probability, we have 
$$
\mu_{n,\ell}\leq e^{-c'(d,k,D)\ell^{2/d}}.
$$
Hence, combined with Theorem \ref{thm:spectralconvergence_heatkernel}, we conclude the claim with $C_{\text{L}}:=c'(d,k,D)$. 
\end{proof}

\subsection{Vector Diffusion Maps and Diffusion Maps}\label{subsection:Appendix:VDMDM}
In this subsection, we discuss a potential application of the CGL and GL -- estimate the local geodesic distance.
Fix $t>0$, we define the {\it vector diffusion maps} (VDM) $V_{t,n}:\mathcal{X}\to \RR^{(nd)^2}$ by
\begin{equation*}
V_{t,n}: x_i \mapsto \left( (\mu_{n,l} \mu_{n,r})^{t} \langle v_{n,l}[i], v_{n,r}[i] \rangle\right)_{l,r=1}^{nd},
\end{equation*}
where $v_{n,l}[i]$ is a $d$-dim vector containing the $((i-1)d+1)$-th entry to the $(id)$-th entry of the eigenvector $v_{n,l}$. 
With this map, the Hilbert-Schmidt norm of the $(i,j)$-th block of $\vdmC_s$ satisfies
\begin{equation*}
\|\vdmC_s^{2t}(i,j)\|^2_{HS} = \langle V_{t,n}(x_i), V_{t,n}(x_j) \rangle,
\end{equation*}
that is, $\|\vdmC_s^{2t}(i,j)\|^2_{HS}$ becomes an inner product for the finite dimensional Hilbert space.
The reason we need to consider $\vdmC_s^{2t}$ but not $\vdmC_s^{t}$ is that all eigenvalues $\mu_{n,l}$ of $\vdmC_s$ reside in the interval $[-1,1]$, and we can not guarantee the positivity of $\mu_{n,l}$ when $n$ is finite. We can then define the {\it vector diffusion distance} (VDD) to quantify the affinity between nodes $i$ and $j$: 
$$
d_{\text{C},t,n}:=\| V_{t,n}(x_i)- V_{t,n}(x_j)\|^2.
$$

The theoretical properties of VDM and VDD will be clear when $n\to \infty$.  
Fix $a\in \mathcal{B}(T\manifold,g)$ and $t>0$, define the following map embedding $x\in\manifold$ to $\ell^2$:
\begin{equation}
V^a_t:x \mapsto \left( \frac{1}{\sqrt{d}(4\pi)^{d/2}t^{(d+1)/2}} e^{-(\lambda_k + \lambda_l)t/2} \langle \VectorFieldOnM_k(x), \VectorFieldOnM_l(x) \rangle \right)_{k,l=1}^\infty.
\end{equation}
With $V^a_t$, we define a new affinity between pairs of points by
\begin{align}
d_{\text{C},t}(x,y) := \|V^a_t(x)-V^a_t(y)\|_{\ell^2}.
\end{align}
Due to Theorem \ref{thm:pointwise_conv_VDM} and Theorem \ref{thm:spectralconvergence_heatkernel}, the VDM (resp. VDD) is a discretization of $V^a_t$ (resp. $d_{\text{C},t}$), so we may abuse the notation and call $V^a_t$ VDM and $d_{\text{C},t}$ VDD. 
We have the following Theorem saying that locally $d_{\text{C},t}$ approximates the geodesic distance:
\begin{thm}[\cite{singer_wu:2012}]\label{thm:geod}
Take a Riemannian manifold $(\manifold,g)$. For all $t>0$, the VDM $V^a_t$ is diffeomorphic. Furthermore, suppose $x,y\in \manifold$ so that $x=\exp_yv$, where $v\in T_y \manifold$. When $\|v\|^2\ll t\ll 1$ we have 
\begin{align}\label{thm:geod_approximation}
d^2_{\text{C},t}(x,y)= \|v\|^2+O(t\|v\|^2).
\end{align}
\end{thm}

Although GL is a special case of CGL, with GL we may define a different embedding which has different features. Given $t>0$ and $0\leq \delta<1$, the 
{\it diffusion maps (DM) with diffusion time $t$}\footnote{In practice, we may consider {\it the truncated diffusion maps (tDM) with diffusion time $t$ and accuracy $\delta$}, which is defined as
$\Phi_{t,n,m(\delta,t)}:\,x_i \mapsto \left( \nu_{n,l}^tu_{n,l}(i) \right)_{l=2}^{m(\delta,t)}$,
where $m(\delta,t)\in\NN$ such that $\lambda_{m(\delta,t)}^t>\delta\lambda_1^t$ and $\lambda_{m(\delta,t)+1}^t\leq\delta\lambda_1^t$ \cite{coifman_lafon:2006}} as 
\begin{equation}
\Phi_{t,n}:\,x_i \mapsto \left( \nu_{n,l}^tu_{n,l}(i) \right)_{l=2}^{n}.
\end{equation}

One similar but different algorithm is the {\it Laplacian eigenmaps} \cite{belkin_niyogi:2003,belkin_niyogi:2005}, that is, $x_i$ is mapped to $\left(u_{n,l}(i) \right)_{l=2}^{m}$, which can be viewed as a special DM with diffusion time $t=0$ and $1<m\leq n$ is chosen by the user.
Yet another similar quantity referred to as the {\it global point signature} proposed in \cite{Rustamov:2007}, which maps $x_i$ to $\left( (-\ln \nu_{n,l} )^{-1/2} u_{n,l}(i) \right)_{l=2}^m$, where $m$ is chosen by the user. Another variation is the {\it commute time embedding}  \cite{Qiu_Hancock:2007}.
We mention in the Laplacian eigenmaps, global point signature and commute time embedding, the notion ``diffusion'' does not exist. Although these mappings are diffeomorphic to each other when $m=n$ via a linear transformation, asymptotically their behaviors are different. Furthermore, even if the connection function is trivial, the VDM and DM are different.
With DM, we introduce a new metric between sampled points, which is referred to as {\it diffusion distance} (DD):
\begin{align}
d_{\textup{DM},t,n}(x_i,\,x_j) := \|\Phi_{t,n}(x_i)-\Phi_{t,n}(x_j)\|_{\RR^{n-1}}.
\end{align}
To study $d_{\textup{DM},t,n}$, we take $a\in \mathcal{B}(\manifold,g)$ and $t>0$, and map $x\in \manifold$ to the Hilbert space $\ell^2$ by \cite{berard_besson_gallot:1994}
\begin{equation}
\Phi^a_t:x \mapsto \sqrt{\text{vol}(\manifold)}\left( e^{-\gamma_\ell t}\phi_\ell(x) \right)_{\ell=1}^\infty,
\end{equation}
With the map $\Phi^a_t$, we are able to define a new affinity between pairs of points:
\begin{align}
d_{\textup{DM},t}(x,y) := \|\Phi^a_t(x)-\Phi^a_t(y)\|_{\ell^2}.
\end{align}
Due to Theorem \ref{thm:pointwise_conv_VDM} and Theorem \ref{thm:spectralconvergence_heatkernel}, the DM (resp. DD) is a discretization of $\Phi^a_t$ (resp. $d_{\text{DM},t}$), so we may abuse the notation and call $\Phi^a_t$ DM and $d_{\text{DM},t}$ DD. 

It has been shown that the DM satisfies the following ``almost isometric'' property \cite{singer_wu:2012}:
\begin{thm}\label{thm:geod:DM}
Take a Riemannian manifold $(\manifold,g)$. For all $t>0$, $\Phi^a_t$ is diffeomorphic. Furthermore, suppose $x,y\in \manifold$ so that $x=\exp_yv$, where $v\in T_y \manifold$. When $\|v\|^2\ll t\ll 1$ we have
\begin{align}\label{thm:geod_approximation}
d^2_{\textup{DM},t}(x,y)= \|v\|^2+O(t\|v\|^2).
\end{align}
\end{thm}
The above theorems, when combined with the above spectral convergence theorem, says that the VDD and DD provide an accurate estimation of the geodesic between two close points. While combined with the manifold sparsity property, we have the following practical fact -- if we are allowed a positive small error when we estimate the geodesic distance, we do not need to recover the whole eigen-structure. Instead, the first few eigenvalues and eigenvectors are enough. 

We have the following statement shown in \cite{Portegies:2013,Bates:2014}. Fix $\epsilon>0$ and $(\manifold,g)$ is a $d$-dim manifold satisfying $\Ric(g)\geq (d-1)kg,\,\mbox{vol}(\manifold)\leq V,\,\text{inj}(\manifold)\leq I$.
Then there exists a $t_0=t_0(d,k,I,\epsilon)$ such that for all $0<t<t_0$, these is $N_E=N_E(d,k,I,V,\epsilon,t)$ so that if $N\geq N_E$, the {\it truncated diffusion maps}
\begin{align}
\Phi^a_{t,N}:\,x\mapsto\sqrt{\text{vol}(\manifold)}\big(e^{-\gamma_\ell t}\phi_\ell(x)\big)_{\ell=1}^N
\end{align}
is an embedding of $\manifold$ into $\RR^N$ and 
\begin{align}
1-\epsilon<\left|\frac{(2t)^{(n+2)/4}\sqrt{2}(4\pi)^{n/4}}{\sqrt{\text{vol}(\manifold)}}\ud \Phi_{N,t}|_x\right|<1+\epsilon.
\end{align}

Before ending this section, we show an interesting example regarding the data visualization and embedding issue. Take the Trefoil knot $\manifold$ embedded  in $\RR^3$ by 
$\iota(t)=[\sin(t)+2\sin(2t),\,\cos(t)-2\cos(2t),\,-\sin(3t)]$, where $t\in [0,2\pi)$.
We refer to Figure \ref{fig5:Trefoil} for an illustration. Note that the Trefoil knot is not homeomorphic to $S^1$.
We sample $1000$ points uniformly from $[0,2\pi)$ independently; that is, we sample $1000$ points on $\manifold$ non-uniformly. If we want to visualize the dataset, we may apply the tDM to embed $\manifold$ to $\RR^3$ (or $\RR^2$). The result is shown in Figure \ref{fig5:Trefoil}. The results deserve some discussion. Note that the tDM maps the Trefoil knot into a circle, which is not homeomorphic to the Trefoil knot; that is, the topology of the Trefoil knot is not preserved. Note that for the visualization purpose, we only choose the first $3$ (or $2$) eigenvectors, which leads to a map which deteriorate the topology. If we want to guarantee the preservation of the topology, we need the embedding theorem counting how many eigenvectors we need. This opens the following question, in particular when the dataset is noisy -- how to balance between different data analysis results, for exampling, how to balance between preserving the topology information and data visualization?

\begin{figure}
\begin{center}
\includegraphics[width=0.5\columnwidth]{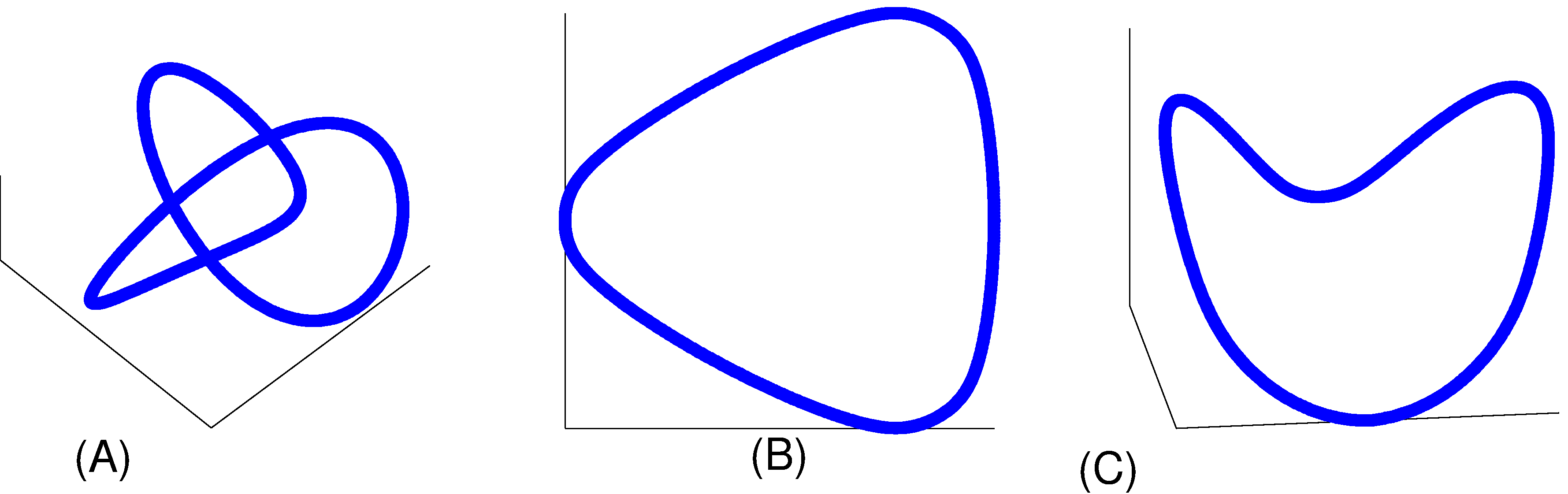}
\end{center}
\caption{\small Left: the Trefoil knot; middle: the truncated DM with the first 2 non-trivial eigenvectors of the GL; right: the truncated DM with the first 3 non-trivial eigenvectors of the GL.}\label{fig5:Trefoil}
\end{figure}

\end{document}